\documentclass[final,1p,times]{elsarticle}

\usepackage[colorlinks=true]{hyperref}
\usepackage{amsfonts}
\usepackage{mathrsfs}
\usepackage{amsthm,amsmath,amssymb,mathrsfs,amsfonts,graphicx,graphics,latexsym,exscale,cmmib57,dsfont,bbm,amscd}
\usepackage{color,soul}
\newtheorem{thm}{Theorem}[section]
\newtheorem{cor}[thm]{Corollary}
\newtheorem{lem}[thm]{Lemma}

\newtheorem*{thm1.4}{Theorem~\ref{thm1.4}}
\newtheorem*{cor1.17}{Corollary~\ref{cor1.17}}

\theoremstyle{definition}
\newtheorem{defn}[thm]{Definition}
\newtheorem*{s-t}{Standing terminologies}
\newtheorem{sn}[thm]{Standing notation}
\newtheorem*{notes}{Notes}
\newtheorem*{note}{Note}

\newcommand{\I}{\boldsymbol{I}}
\journal{arXiv:1806.09305 [math.DS]; CJM April 16, 2019}

\begin{document}

\begin{frontmatter}
\title{The McMahon pseudo-metrics of minimal semiflows with invariant measures}

\author{Xiongping Dai}
\ead{xpdai@nju.edu.cn}
\address{Department of Mathematics, Nanjing University, Nanjing 210093, People's Republic of China}


\begin{abstract}
Using McMahon pseudo-metrics, for any minimal semiflow admitting an invariant measure, we study the relationships between its equicontinuous structure relation, regionally proximal relation and Veech's relations; and characterize its weak-mixing. We show its Veech Structure Theorem if it is almost automorphic.
%
%
\end{abstract}

\end{frontmatter}

\section{Introduction}\label{sec1}
As usual, a \textit{semiflow} with phase space $X$ and with phase semigroup $T$ is understood as a pair $(T,X)$ where, unless specified otherwise, we assume that
\begin{itemize}
\item $X$ is a compact $\textrm{T}_2$-space, \textbf{not necessarily metrizable}, and $T$ a topological \textbf{semigroup} with the identity $e$; moreover, $T$ acts jointly continuously on $X$ by phase mapping $(t,x)\mapsto tx$ such that $ex=x$ and $(st)x=s(tx)$ for all $x\in X$ and $s,t\in T$.
\end{itemize}
When $T$ is a topological group here, $(T,X)$ will be called a \textit{flow}.
\begin{s-t}
Given $x\in X, U\subseteq X$, and $V\subseteq X$, for $(T,X)$, if no confusion, we write for convenience
\begin{enumerate}
\item $N_T(x,U)=\{t\,|\,tx\in U\}$ and $N_T(U,V)=\{t\,|\,U\cap t^{-1}V\not=\emptyset\}$.
\end{enumerate}
We shall say that:
\begin{enumerate}
\item[2.] $(T,X)$ is \textit{minimal} if and only if $Tx$ is dense in $X$: $\mathrm{cls}_XTx=X$, for all $x\in X$; An $x\in X$ is called a \textit{minimal point} or an \textit{a.p. point} of $(T,X)$ if $\mathrm{cls}_XTx$ is a minimal subset of $X$.
\item[3.] $(T,X)$ admits of an \textit{invariant measure} $\mu$ if $\mu$ is a Borel probability measure on $X$ such that $\mu(B)=\mu(t^{-1}B)$ for all Borel set $B\subseteq X$ and each $t\in T$. If every non-empty open subset of $X$ is of positive $\mu$-measure, then we say $\mu$ is of \textit{full support}.

\item[4.] $(T,X)$ is \textit{surjective} if for each $t$ of $T$, $x\mapsto tx$ is an onto self-map of $X$.
\item[5.] $(T,X)$ is \textit{invertible} if for each $t\in T$, $x\mapsto tx$ is a 1-1 onto self-map of $X$. In this case, by $\langle T\rangle$ we mean the smallest group of self-homeomorphisms of $X$ with $T\subseteq\langle T\rangle$.

\item[6.] $T$ is called \textit{amenable} if any semiflow on a compact $\textrm{T}_2$-space with the phase semigroup $T$ admits an invariant measure. Particularly, each abelian semigroup is amenable.
\end{enumerate}
Let $\varDelta_X=\{(x,x)\,|\,x\in X\}$ and $\mathscr{U}_X$ the compatible symmetric uniform structure of $X$. For $\varepsilon\in\mathscr{U}_X$ and $x\in X$, set
\begin{enumerate}
\item[7.] $\varepsilon[x]=\{y\in X\,|\,(x,y)\in\varepsilon\}$, which is a neighborhood of $x$ in $X$.
\end{enumerate}

Recall that a set-valued map $f\colon X\rightsquigarrow X$ is said to be \textit{continuous} at $x_0\in X$ if and only if given $\varepsilon\in\mathscr{U}_X$ there is a $\delta\in\mathscr{U}_X$ such that $x\in\delta[x_0]$ implies $f(x)\subseteq\varepsilon[f(x_0)]$ and $f(x_0)\subseteq\varepsilon[f(x)]$.
\begin{enumerate}
\item[8.] A surjective semiflow $(T,X)$ is called \textit{bi-continuous} if for each $t\in T$, $x\rightsquigarrow t^{-1}x$ is continuous at every point of $X$.
\end{enumerate}
Clearly, if $(T,X)$ is invertible, it is bi-continuous. If let $X=\{0,1\}^{\mathbb{Z}_+}$ with the usual topology and $\sigma\colon X\rightarrow X$ defined by $(x_i)_{i\in\mathbb{Z}_+}\mapsto(x_{i+1})_{i\in\mathbb{Z}_+}$, then the cascade $(\sigma,X)$ is bi-continuous. More generally, if $(T,X)$ is such that for each $t\in T$, $x\mapsto tx$ is an $N$-to-1 locally homeomorphism of $X$, then $(T,X)$ is a bi-continuous semiflow.
\end{s-t}

In this paper, we shall study the equicontinuous structure relation, weak-mixing and weak disjointness, and Veech's relations by using the McMahon pseudo-metric of minimal (bi-continuous) semiflows with invariant measures.

\subsection{Equicontinuous structure relation}\label{sec1.1}
First of all we recall a basic notion\,---\,equicontinuity, which is valid for any semiflow $(T,X)$ and for which there are some equivalent conditions in \cite{AD}.

\begin{defn}\label{def1.1}
$(T,X)$ is \textit{equicontinuous} if and only if given $\varepsilon\in\mathscr{U}_X$, there exists some $\delta\in\mathscr{U}_X$ such that $T\delta\subseteq\varepsilon$, i.e., if $(x,x^\prime)\in\delta$ then $(tx,tx^\prime)\in\varepsilon\ \forall t\in T$.
\end{defn}

As is proved in \cite{EG, V77, Aus}, for $(T,X)$ there always exists on $X$ a least closed $T$-invariant equivalence relation, denoted $S_{\textit{eq}}(X)$, such that $(T,X/S_{\textit{eq}}(X))$, defined by
\begin{gather*}
(t,S_\textit{eq}[x])\mapsto S_\textit{eq}[tx]\quad \forall (t,x)\in T\times X,
\end{gather*}
is an equicontinuous semiflow. Then:
\begin{defn}\label{def1.2}
Based on a semiflow $(T,X)$:
\begin{enumerate}
\item[1.2a] $S_{\textit{eq}}(X)$ is called the \textit{equicontinuous structure relation} of $(T,X)$. Simply write $X_{\textit{eq}}=X/S_{\textit{eq}}(X)$ if no confusion.

\item[1.2b] $(T,X_{\textit{eq}})$ is called the \textit{maximal equicontinuous factor} of $(T,X)$.

\item[1.2c] Write $\pi\colon X\rightarrow X_{\textit{eq}}$, viz $x\mapsto S_{\textit{eq}}[x]$, for the canonical projection.
\end{enumerate}
\end{defn}

In a number of situations the equicontinuous structure relation of minimal flows is known explicitly (cf.,~e.g.,~\cite{V68, Aus, AEE, Aus01, AGN}). Particularly, if $Q$ is the ``regionally proximal relation'' (cf.~Definition~\ref{def1.3} below), then the following two facts are well known:
\begin{enumerate}
\item[]\textbf{Theorem~A} (cf.~\cite[Theorem~2.6.2]{V77}).~ {\it Let $(T,X)$ be a minimal flow such that $(T,X\times X)$ has a dense set of minimal points; then $S_\textit{eq}(X)=Q(X)$.} (This is due to R.~Ellis and W.A.~Veech independently.)
\end{enumerate}

\begin{enumerate}
\item[]\textbf{Theorem~B} (cf.~McMahon~\cite{McM} and also see \cite[Theorem~9.8]{Aus}).~ {\it If $(T,X)$ is a minimal flow admitting an invariant measure, then $S_{\textit{eq}}(X)=Q(X)$.}
\end{enumerate}

Based on the recent work~\cite{DX,AD,Dai}, we will generalize the above classical Theorem~B to semiflows on compact $\textrm{T}_2$-spaces; see Theorem~\ref{thm1.4} below.

\begin{defn}\label{def1.3}
Let $(T,X)$ be a semiflow with phase semigroup $T$.
\begin{enumerate}
\item $(T,X)$ is called \textit{distal} if given $x,x^\prime\in X$ with $x\not=x^\prime$, there is an $\varepsilon\in\mathscr{U}_X$ such that $t(x,x^\prime)\not\in\varepsilon$ for all $t\in T$. Then by \cite[Theorem~1.15]{AD}, whenever $(T,X)$ is distal, then it is invertible.

\medskip
\item[] \textbf{Theorem~C} (cf.~\cite{AD}).~ {\it If $(T,X)$ is a (minimal) distal semiflow, then $(\langle T\rangle,X)$ is a (minimal) distal flow. Hence by Furstenberg's theorem~\cite{F63}, $(T,X)$ admits invariant measures.}

\medskip
\item An $x\in X$ is called a \textit{distal point} of $(T,X)$ if no point of $\textrm{cls}_XTx$ other than $x$ is proximal to $x$; i.e., if $y\in\textrm{cls}_XTx$ is such that $t_n(x,y)\to(z,z)$ for some $z\in X$ and some net $\{t_n\}$ in $T$ then $y=x$. It is easy to verify that:

\medskip
\item[]\textbf{Lemma~D}.~ {\it $(T,X)$ is distal iff every point of $X$ is distal for $(T,X)$.}

\medskip
$(T,X)$ is called \textit{point-distal} if there is a distal point whose orbit is dense in $X$.

\item We say $x\in X$ is \textit{regionally proximal to} $x^\prime\in X$ for $(T,X)$, denoted by $(x,x^\prime)\in Q(X)$ or $x^\prime\in Q[x]$, if there are nets $\{x_n\}, \{x_n^\prime\}$ in $X$ and $\{t_n\}$ in $T$ with $x_n\to x, x_n^\prime\to x^\prime$ and $\lim_n t_nx_n=\lim_nt_nx_n^\prime$.
\end{enumerate}
\end{defn}

Then by the equality:
\begin{gather*}
Q(X)={\bigcap}_{\alpha\in\mathscr{U}_X}\mathrm{cls}_{X\times X}{\bigcup}_{t\in T}t^{-1}\alpha
\end{gather*}
$Q(X)$ is a closed reflexive symmetric relation. An example due to McMahon~\cite{M76} shows it is not always an equivalence relation even in minimal flows.

If $(T,X)$ is a flow and $(x,x^\prime)\in Q(X)$, then for all $t\in T$, by $\lim_n(t_nt^{-1})(tx_n)=\lim_n(t_nt^{-1})(tx_n^\prime)$ and $t_nt^{-1}\in T$, it easily follows that $t(x,x^\prime)\in Q(X)$ and so $Q(X)$ is invariant.

In general, $Q(X)$ is neither invariant nor transitive in semiflows; yet if $(T,X)$ is minimal admitting an invariant measure, as we will see, $Q(X)$ is an invariant closed equivalence relation on $X$, and in fact

\begin{thm}\label{thm1.4}
Let $(T,X)$ be a minimal semiflow, which admits an invariant measure on $X$, then $S_{\textit{eq}}(X)=Q(X)$. Hence $(T,X_\textit{eq})$ is equicontinuous invertible.
\end{thm}

It should be noted that the condition ``admitting an invariant measure'' is important for our statement above. For example, let $X=[0,1]$ the unit interval with the usual topology and for every $\alpha$ with $0<\alpha<1$, define two injective mappings of $X$ into itself:
\begin{gather*}
f_\alpha\colon X\rightarrow X,\ x\mapsto\alpha x\quad {\textrm{and}}\quad
g_\alpha\colon X\rightarrow X,\  x\mapsto 1-\alpha(1-x).
\end{gather*}
Now let
$$
T=\left\{\left.f_{\alpha_1}^{\epsilon_1}g_{\gamma_1}^{\delta_1}\dotsm f_{\alpha_n}^{\epsilon_n}g_{\gamma_n}^{\delta_n}\,\right|\,0<\alpha_i<1, 0<\gamma_i<1,\ \epsilon_i=0\textrm{ or } 1,\ \delta_i=0\textrm{ or }1,\ n=1,2,\dotsc\right\}
$$
be the discrete semigroup generated by $f_\alpha,g_\alpha$, $0<\alpha<1$. It is easy to see that each $t\in T$ is injective and that $(T,X)$ is equicontinuous minimal so $S_{\textit{eq}}(X)=\varDelta_X$, but $Q(X)=X\times X$. See \cite[Example~1.2]{AD}.

Thus, as consequences of Theorem~\ref{thm1.4}, we can easily obtain the following two corollaries.

\begin{cor}\label{cor1.5}
Let $(T,X)$ be a minimal semiflow with $T$ amenable; then $S_{\textit{eq}}(X)=Q(X)$.
\end{cor}

Thus if $f\colon X\rightarrow X$ is a minimal continuous transformation of $X$, then $S_{\textit{eq}}(X)=Q(X)$ associated to the natural $\mathbb{Z}_+$-action.

Next, replacing amenability of $T$ in Corollary~\ref{cor1.5} by distality of $(T,X)$, we can obtain the following result:

\begin{cor}\label{cor1.6}
Let $(T,X)$ be a minimal distal semiflow; then $S_{\textit{eq}}(X)=Q(X)$.
\end{cor}

\begin{proof}
Since $(T,X)$ is distal, each $t\in T$ is a self-homeomorphism of $X$. Then by Theorem~C, $(T,X)$ admits an invariant measure. Then Corollary~\ref{cor1.6} follows from Theorem~\ref{thm1.4}.
\end{proof}

$\S\S\ref{sec2}$ and \ref{sec3} of this paper will be devoted to proving Theorem~\ref{thm1.4}. Notice that although we will see that $(T,X_\textit{eq})$ is invertible in the situation of Theorem~\ref{thm1.4}, yet each $t\in T$ itself need not be invertible with respect to $(T,X)$. This will cause the difficulty.

\subsection{Weak-mixing semiflows}\label{sec1.2}
Let $(T,X)$ be a semiflow. Let $X^n=X\times\dotsm\times X$ ($n$-times), for any integer $n\ge2$. Based on $(T,X)$, $(T,X^n)$ is also a semiflow, which is defined by $t\colon (x_i)_{i=1}^n\mapsto(tx_i)_{i=1}^n$. Given any cardinality $\mathfrak{c}\ge1$, we can similarly define $(T,X^\mathfrak{c})$.
By $\mathfrak{U}(X)$ we denote the collection of non-empty open subsets of $X$.
Then we introduce some basic standard notions.
\begin{enumerate}
\item[1)] $(T,X)$ is called \textit{topologically transitive} (T.T.) iff every $T$-invariant set with non-empty interior is dense in $X$ iff $N_T(U,V)\not=\emptyset$ for $U,V\in\mathfrak{U}(X)$.

\item[2)] $(T,X)$ is called \textit{point-transitive} (P.T.) if there is some point $x\in X$ with $Tx$ dense in $X$. In this case, write $x\in \textrm{Tran}\,(T,X)$.

\item[3)] $(T,X)$ is called \textit{weak-mixing} if $(T,X\times X)$ is a T.T. semiflow.

\item[4)] $(T,X)$ is called \textit{discretely thickly T.T.} if given $U,V\in\mathfrak{U}(X)$, $N_T(U,V)$ is \textit{discretely thick} (i.e. for any finite subset $F$ of $T$ there is $t\in T$ such that $Ft\subseteq N_T(U,V)$).

\item[5)] $A\subseteq T$ is called an \textit{IP-set} of $T$ if there exists a sequence $\{t_n\}$ in $T$ such that $t_{n_1}t_{n_2}\dotsc t_{n_k}\in A$ for all $1\le n_1<n_2<\dotsm<n_k<\infty$ and $k\ge1$. $(T,X)$ is said to be \textit{IP-T.T.} if $N_T(U,V)$ is an IP-set of $T$ for $U,V\in\mathfrak{U}(X)$.

\item[6)] Let $n\ge2$ be an integer. We say $(T,X)$ satisfies the \textit{$n$-order Bronstein condition} if $X^n$ contains a dense set of a.p. points. We say $(T,X)$ and $(T,Y)$ satisfies the \textit{Bronstein condition} if the a.p. points are dense in $X\times Y$.
\end{enumerate}

In \cite[p.~26]{F67}, T.T. is called ``ergodic''. It should be noted that in our situation, T.T. $\not\Rightarrow$ P.T., and weak-mixing $\not\Rightarrow$ P.T. as well.

In $\S\ref{sec4}$, we will characterize ``weak-mixing'' of minimal semiflow with invariant measures by using the McMahon pseudo-metric; see Theorem~\ref{thm4.12}.

It is a well-known important fact that (cf.~\cite[Proposition~II.3]{F67} for cascades, \cite[Theorem~1.11]{G76} for $T$ abelian groups, and \cite[Lemma~3.2]{DT} for abelian semigroups):
\begin{quote}
\textbf{Theorem~E}.~ \textit{If $(T,X)$ is weak-mixing with $T$ abelian, then $(T,X^n)$ is T.T. for all integer $n\ge2$}.
\end{quote}

In fact, this very important theorem can be extended to amenable semigroups as follows, which is new even for flows.

\begin{thm}\label{thm1.7}
Let $(T,X)$ be a minimal semiflow, which admits of an invariant measure $\mu$ $($e.g. $T$ is an amenable semigroup$)$. Then $(T,X)$ is weak-mixing iff $(T,X\times Y)$ is T.T. for all $(T,Y)$ which is T.T. and admits of a full ergodic center iff $Q(X)=X\times X$.
\end{thm}
\begin{note}
This theorem is comparable with the following theorem in which our ergodic condition ``$\mu$'' and ``full ergodic center'' is replaced by the topological condition---``Bronstein condition'':
\begin{quote}
\textbf{Theorem~F} (cf.~Veech~\cite[Theorem~2.6.3 for the relativized version]{V77}).~ \textit{Let $T$ be a group. Let $(T,X)$ be minimal which satisfies the 2-order Bronstein condition, and
suppose $S_\textit{eq}(X)=X\times X$. If $(T,Z)$ is T.T. such that $(T,X)$ and $(T,Z)$ satisfies the Bronstein condition, then $(T,X\times Y)$ is a T.T. flow}.
\end{quote}
Theorem~F implies that
\begin{itemize}
\item {\it Let $(T,X)$ be a minimal flow satisfying the 2-order Bronstein condition. Then $(T,X)$ is weak-mixing iff $Q(X)=X\times X$.}
\end{itemize}
\end{note}
The ``ergodic center'' will be defined in Definition~\ref{def4.3}. This theorem will be proved in $\S$\ref{sec4.1} under the guises of Corollary~\ref{cor4.8} for the first ``if and only if'' and Theorem~\ref{thm4.12} for the second ``if and only if''. By induction based on Theorem~\ref{thm1.7} we can obtain the following

\begin{thm}\label{thm1.8}
Let $(T,X)$ be a minimal semiflow, which admits an invariant measure $($e.g. $T$ an amenable semigroup$)$. Then $(T,X)$ is weak-mixing iff $(T,X^\mathfrak{c})$ is T.T. for all cardinality $\mathfrak{c}\ge2$ iff $(T,X)$ is a discretely thickly T.T. semiflow.
\end{thm}

\begin{note}
Similar to Theorem~\ref{thm1.7}, using the Bronstein condition instead of our ergodic condition we can obtain the following
\begin{quote}
\textbf{Theorem~G} (cf.~Veech~\cite[Theorem~2.6.4 for the relativized version]{V77}).~ \textit{Let $(T,X)$ be a minimal flow which satisfies the $n$-order Bronstein condition for each $n\ge2$. Then $(T,X^n)$ is T.T. for all $n\ge2$}.
\end{quote}
\end{note}

Theorem~\ref{thm1.8} will be proved in $\S\ref{sec4.1}$ under the guise of Theorem~\ref{thm4.11}. It is not known whether these assertions still hold for a non-minimal semiflow with amenable phase semigroup.

\begin{defn}[\cite{Dai}]
An $x\in X$ is called \textit{locally almost periodic} (l.a.p.) for $(T,X)$ if for each neighborhood $U$ of $x$ there are a neighborhood $V$ of $x$ and a discretely syndetic subset $A$ of $T$ such that $AV\subseteq U$.

Here by a ``discretely syndetic'' set $A$ we mean that there is a finite subset $K$ of $T$ such that $Kt\cap A\not=\emptyset\ \forall t\in T$. A subset of $T$ is discretely syndetic if and only if it intersects non-voidly each discretely thick subset of $T$ (cf., e.g.,~\cite{AD}).
\end{defn}

\begin{cor}\label{cor1.10}
Let $(T,X)$ be a minimal semiflow with $T$ amenable. If $(T,X)$ is weak-mixing, then there is no l.a.p. point of $(T,X)$.
\end{cor}

\begin{proof}
Suppose the contrary that there is an l.a.p. point $x_0$. Since $X$ is a non-singleton $\textrm{T}_2$-space, we can choose disjoint $U,U^\prime\in\mathfrak{U}(X)$ with $x_0\in U$. Then there are a discretely syndetic set $A\subseteq T$ and an open neighborhood $V$ of $x$ such that $AV\subseteq U$. But this contradicts that $N_T(V,U^\prime)$ is discretely thick in $T$ by Theorem~\ref{thm1.8}. This proves Corollary\ref{cor1.10}.
\end{proof}

Recall that an $x\in X$ is referred to as an \textit{IP$^*$-recurrent point} of $(T,X)$ if and only if given a neighborhood $U$ of $x$, $N_T(x,U)$ is an \textit{IP$^*$-set} in the sense that it intersects non-voidly every IP-set of $T$ (cf.~\cite{Fur, DL}).

\begin{quote}
\textbf{Theorem~H} (cf.~\cite[Theorem~9.11]{Fur} and \cite[Theorem~4]{DL}).~ {\it Let $(T,X)$ be any semiflow and $x\in X$; then $x$ is a distal point of $(T,X)$ if and only if $x$ is an IP$^*$-recurrent point of $(T,X)$.}
\end{quote}

\begin{cor}\label{cor1.11}
If $(T,X)$ is a minimal weak-mixing semiflow with $T$ amenable, then there exists no distal point of $(T,X)$.
\end{cor}

\begin{notes}
$\,$
\begin{enumerate}
\item See Veech~\cite{V70} and Furstenberg~\cite[Theorem~9.12]{Fur} for cascade $(f,X)$ with $X$ a compact metric space and Dai-Tang~\cite[Proposition~3.5]{DT} for IP-T.T. semiflow $(T,X)$ with first countable phase space $X$. Here without any countability our proof is completely new.

\item If $T$ is a group in place of ``$T$ ameanble'', the consequence of Corollary~\ref{cor1.11} still holds using Veech's Structure Theorem of point-distal flows; see Theorem~\ref{thm2.10}.

\item By Corollary~\ref{cor1.11}, a point-distal semiflow is T.T. but it is never an IP-T.T. semiflow.
\end{enumerate}
\end{notes}

\begin{proof}
Assume for a contradiction that $(T,X)$ has a distal point, say $x_0\in X$. Let $W$ be a closed subset of $X$ with $x_0\not\in W$ and $\textrm{Int}_XW\not=\emptyset$. Let $\beta\in\mathscr{U}_X$ be such that $\beta[x_0]\cap W=\emptyset$. For all $\varepsilon\in\mathscr{U}_X$ with $\varepsilon\le\beta$, define a closed subset of $X$ as follows:
$W_\varepsilon=\left\{x\in W\,|\,\exists\,t\in T\textit{ s.t. }tx_0,tx\in\textrm{cls}_X\varepsilon[x_0]\right\}$.
Since $x_0$ is IP$^*$-recurrent by Theorem~H, $N_T(x_0,\varepsilon[x_0])$ is an IP$^*$-set. On the other hand, by Theorem~\ref{thm1.8}, $N_T(W,\varepsilon[x_0])$ is discretely thick and so $N_T(W,\varepsilon[x_0])$ is an IP-set (cf., e.g., \cite[Lemma~9.1]{Fur} and \cite[(1.2a)]{Dai}). Therefore, $W_\varepsilon\not=\emptyset$ for all $\varepsilon\in\mathscr{U}_X$. Given any $\varepsilon_1,\dotsc,\varepsilon_n\le\beta$, let $\varepsilon\le\varepsilon_1\cap\dotsm\cap\varepsilon_n$; then $W_{\varepsilon_1}\cap\dotsm\cap W_{\varepsilon_n}\supseteq W_\varepsilon\not=\emptyset$. Thus $\{W_\varepsilon\,|\,\varepsilon\le\beta\}$ has the finite intersection property. This shows that $\bigcap_{\varepsilon\le\beta}W_\varepsilon\not=\emptyset$. Now let $y\in\bigcap_{\varepsilon\le\beta}W_\varepsilon$ be any given. Clearly, $x_0\not=y$. Since for each $\varepsilon\le\beta$ there is some $t=t_\varepsilon\in T$ such that $tx_0, ty\in\textrm{cls}_X\varepsilon[x_0]$, hence $y$ is proximal to $x_0$. This is a contradiction to the assumption that $x_0$ is a distal point of $(T,X)$. This proves Corollary~\ref{cor1.11}.
\end{proof}

\begin{thm}[{Vecch 1977~\cite[Question, p.~802]{V77} for any point-distal flows}]\label{thm1.12}
Let $(T,X)$ be a point-distal non-trivial semiflow with $T$ amenable. Then $(T,X)$ has a non-trivial equicontinuous factor.
\end{thm}

\begin{proof}
By Theorem~\ref{thm1.4}, $S_\textit{eq}(X)=Q(X)$. Suppose the contrary that $(T,X)$ has no non-trivial equicontinuous factor, then $Q(X)=X\times X$ and so $(T,X)$ is weakly mixing by Theorem~\ref{thm1.7}. But this contradicts Corollary~\ref{cor1.11}. The proof is complete.
\end{proof}

It should be noticed that ``$T$ amenable'' may be relaxed by ``$(T,X)$ admitting of an invariant measure'' in Corollaries~\ref{cor1.10} and \ref{cor1.11} and Theorem~\ref{thm1.12}.

As another application of Theorem~\ref{thm1.8} we will consider the chaotic dynamics of weak-mixing semiflows with amenable phase semigroups in $\S\ref{sec4.2}$; see Theorems~\ref{thm4.17} and \ref{thm4.19} and Corollary~\ref{cor4.21}.

\subsection{Veech's relations of surjective dynamics}\label{sec1.3}
\begin{defn}[{cf.~\cite{V65,V68, AM, AGN} for $T$ in groups and \cite{Dai} for any semiflows}]\label{def1.13}
Let $(T,X)$ be any semiflow.
\begin{enumerate}
\item[(1)] An $x\in X$ is called \textit{almost automorphic} (a.a.) for $(T,X)$, denoted $x\in P_{\!aa}(T,X)$, if and only if
$t_nx\to y, x_n^\prime\to x^\prime, t_nx_n^\prime=y$ implies $x=x^\prime$ for every net $\{t_n\}$ in $T$.

\item[(2)] If $x\in P_{\!aa}(T,X)$ and $\mathrm{cls}_XTx=X$, then $(T,X)$ is called an \textit{a.a. semiflow}.

\item[(3)] We say that $(x,x^\prime)$ is in Veech's relation $V$ of $(T,X)$, denoted $(x,x^\prime)\in V(X)$ or $x^\prime\in V[x]$, if there exist a net $x_n^\prime\to x^\prime$, a point $y\in X$, and a net $\{t_i\}$ in $T$ such that $t_nx\to y$ and $t_nx_n^\prime=y$. Then:
    \begin{itemize}
    \item $V[x]=\{x\}$ if and only if $x\in P_{\!aa}(T,X)$.
    \end{itemize}
\item[(4)] Given $x\in X$, define $D[x]$ for $(T,X)$ by
$$
D[x]={\bigcap}_{\varepsilon\in\mathscr{U}_X}D(x,\varepsilon),\quad \textrm{where }D(x,\varepsilon)=\mathrm{cls}_XA^{-1}Ax\textrm{ and }A=N_T(x,\varepsilon[x]).
$$
$D[x]$ is closed, and, of course, $x\in D[x]$. We will say $(x,y)\in D(X)$ if and only if $y\in D[x]$. Then
\begin{itemize}
    \item $D[x]=\{y\,|\,\exists\{t_n\}, \{s_n\}\textrm{ in }T, \{y_n\}\textrm{ in }X\textrm{ s.t. }(t_nx,s_nx,y_n)\to(x,x,y), t_ny_n=s_nx\}$.
    \end{itemize}
\begin{note}
$D(x,\varepsilon)$ is originally defined as $D(x,\varepsilon)=\mathrm{cls}_XAA^{-1}x$ by Veech in \cite{V68}. However, since $T$ is an abelian group in \cite{V68}, so our definition here agrees with Veech's in flows with abelian phase groups.
\end{note}
\end{enumerate}

Since an a.a. point is a distal point of any surjective $(T,X)$ (cf.~\cite[Lemma~1.8]{Dai} and \cite[Theorem~9.13]{Fur}), hence an a.a. surjective semiflow is point-distal and so minimal. We shall consider another Veech relation $U(X)$ in $\S$\ref{sec3.3}.
\end{defn}

\begin{defn}\label{def1.14}
Let $(T,X)$ and $(T,Y)$ be two any semiflows.
\begin{enumerate}
\item $\pi\colon X\rightarrow Y$ is called an \textit{epimorphism} between $(T,X)$ and $(T,Y)$, denoted $(T,X)\xrightarrow{\pi}(T,Y)$, if $\pi$ is continuous surjective with $\pi(tx)=t\pi(x)$ for all $t\in T$ and $x\in X$. In this case, $(T,Y)$ is called a \textit{factor} of $(T,X)$, and $(T,X)$ an \textit{extension} of $(T,Y)$.
\item $(T,X)\xrightarrow{\pi}(T,Y)$ is said to be of \textit{almost 1-1 type} if there exists a point $y\in Y$ such that $\pi^{-1}(y)$ is a singleton set.
\end{enumerate}
\end{defn}

It is clear that if $(T,X)$ is surjective (resp.~minimal), then its factors are also surjective (resp.~minimal). Of course, its factor need not be invertible if $(T,X)$ is invertible. Moreover, even if $(T,X)$ has an invertible non-trivial factor, $(T,X)$ itself need not be invertible.
As mentioned before, we will be mainly concerned with the maximal equicontinuous factor $(T,X_\textit{eq})$ of a minimal semiflow $(T,X)$.

Using a variation of a theorem of Bogoliouboff and F{\o}lner \cite[Theorem~4.1]{V68} that is valid only for discrete abelian groups, Veech proved the following theorem:
\begin{quote}
\textbf{Theorem~I} (cf.~\cite[Theorem~1.1]{V68}).~ {\it If $(T,X)$ is a minimal flow with $T$ an abelian group, then $D(X)=S_\textit{eq}(X)$.}
\end{quote}

Although Veech's proof of the above Theorem~I does not work for non-abelain flows, yet using different approaches\,---\,McMahon pseudo-metrics, we can obtain the following generalization in $\S$\ref{sec3} and $\S$\ref{sec5}.

\begin{thm}[{Veech's Structure Theorem for a.a. semiflows; cf.~\cite[Theorem~5.9]{Dai} for $T$ abelian}]\label{thm1.15}
Let $(T,X)$ be a minimal bi-continuous semiflow, which admits an invariant measure. Then:
\begin{enumerate}
\item[$(1)$] $D[x]=Q[x]$ for all $x\in X$.
\item[$(2)$] When $(T,X)$ is invertible, then $Q[x]=D[x]=\overline{V[x]}$ for all $x\in X$.
\item[$(3)$] $(T,X)$ is $\mathrm{a.a.}$ if and only if $\pi\colon(T,X)\rightarrow(T,X_\textit{eq})$ is of almost 1-1 type.
\end{enumerate}
Hence if $X$ is compact metric and $(T,X)$ is a.a., then $P_{\!aa}(T,X)$ is a residual subset of $X$.
\end{thm}

\begin{cor}[{Veech for $T$ in amenable groups; unpublished notes}]\label{cor1.16}
Let $(T,X)$ be an invertible minimal semiflow with $T$ an amenable semigroup. Then:
\begin{enumerate}
\item[$(1)$] $Q[x]=D[x]=\overline{V[x]}$ for all $x\in X$.
\item[$(2)$] $(T,X)$ is a.a. if and only if $\pi\colon(T,X)\rightarrow(T,X_\textit{eq})$ is of almost 1-1 type.
\end{enumerate}
\end{cor}

\begin{proof}
Since $T$ is amenable, $(T,X)$ admits an invariant measure. Then Corollary~\ref{cor1.16} follows from Theorem~\ref{thm1.15}.
\end{proof}

\begin{cor}\label{cor1.17}
Let $(T,X)$ be a minimal bi-continuous semiflow, which admits an invariant measure. Then $(T,X)$ is equicontinuous iff all points are a.a. points.
\end{cor}

This will be proved in $\S\ref{sec5}$. We note that Corollary~\ref{cor1.17} in the case that $T$ is an abelian group and $X$ a compact metric space is \cite[Corollary~21]{AM}. Moreover, the statement of Corollary~\ref{cor1.17} still holds without the condition ``which admits an invariant measure'' by using different approaches (cf.~\cite[Theorem~4.5]{Dai}).

\begin{cor}\label{cor1.18}
Let $(T,X)$ be a minimal bi-continuous semiflow, which admits an invariant measure. If $(T,X)$ is a.a., then it has l.a.p. points.
\end{cor}

\begin{proof}
By Theorem~\ref{thm1.15}, $\pi\colon X\rightarrow X_\textit{eq}$ is of almost 1-1 type. Since $(T,X_\textit{eq})$ is equicontinuous invertible, it is l.a.p. and so $(T,X)$ has l.a.p. points. The proof is complete.
\end{proof}

Theorem~\ref{thm1.15} will be proved in $\S$\ref{sec3} and $\S$\ref{sec5}. In fact, (1) of Theorem~\ref{thm1.15} follows from Theorem~\ref{thm3.10} and Corollary~\ref{cor3.13}; (2) of Theorem~\ref{thm1.15} is Theorem~\ref{5.3}; and (3) of Theorem~\ref{thm1.15} is Theorem~\ref{thm5.5}. In fact, since $(T,X)$ is bi-continuous, $(T,X)$ must be invertible when it is a.a by \cite[Lemma~1.1]{Dai}.
\section{Preliminary lemmas}\label{sec2}
To prove our main results Theorem~\ref{thm1.4}, Theorem~\ref{thm1.7} and Theorem~\ref{thm1.15}, we will need some preliminary lemmas and theorems. Among them, Theorem~\ref{thm2.7A} asserts that an epimorphism can transfer $Q(X)$ of a minimal semiflow $(T,X)$ onto the regionally proximal relation of its invertible factor. And Theorem~\ref{thm2.10} is a supplement of Corollary~\ref{cor1.11}.

\subsection{Preliminary notions}
Let $(T,X)$ be a semiflow with phase semigroup $T$ and with compact $\textrm{T}_2$ phase space $X$. We have introduced some necessary notions in $\S\ref{sec1}$. Here we need to introduce another one.

\begin{defn}\label{def2.1}
We shall say that $(T,X)$ admits an \textit{invariant quasi-regular Borel probability measure} $\mu$, provided that $\mu$ is an invariant Borel probability measure on $X$ such that
\begin{itemize}
\item $\mu$ is \textit{quasi-regular} in the sense that for any Borel set $B$ and $\varepsilon>0$ one can find an open set $U$ with $B\subseteq U$ and $\mu(U-B)<\varepsilon$, and for any open set $U$ and $\varepsilon>0$ one can find a compact set $K$ with $K\subset U$ and $\mu(U-K)<\varepsilon$.
\end{itemize}

Since $X$ is compact $\textrm{T}_2$ here, a quasi-regular Borel probability measure must be regular. Moreover, it is well known that if $T$ is an amenable semigroup or $(T,X)$ is distal, then $(T,X)$ always admits an invariant Borel probability measure.
\end{defn}

Recall that a point is called \textit{minimal} if its orbit closure is a minimal subset.
Since $X$ is compact, then by Zorn's lemma, minimal points always exist. A minimal point is also called an ``almost periodic point'' or a ``uniformly recurrent point'' in some works like \cite{E69, Fur, Aus}. If $(T,X)$ is a minimal semiflow, then $\varDelta_X$ is a minimal set of $(T,X\times X)$.

If the set of minimal points is dense in $X$, then we will say $(T,X)$ has a \textit{dense set of minimal points} or \textit{with dense almost periodic points}.

\subsection{Preliminary lemmas}
Now we will introduce and establish in this subsection some preliminary lemmas needed in our later discussion.

\begin{lem}[\cite{AD}]\label{lem2.2}
Let $(T,X)$ be a semiflow. Then:
\begin{enumerate}
\item[$(1)$] If $(T,X)$ is distal, then it is invertible.
\item[$(2)$] If $(T,X)$ is surjective equicontinuous, then it is distal.
\item[$(3)$] If $(T,X)$ is invertible equicontinuous, then $(\langle T\rangle,X)$ is an equicontinuous flow.
\end{enumerate}
\end{lem}

\begin{lem}[\cite{DX,AD}]\label{lem2.3}
$(T,X)$ is an equicontinuous surjective semiflow if and only if $Q(X)=\varDelta_X$ if and only if $(\langle T\rangle,X)$ is an equicontinuous flow.
\end{lem}

\begin{lem}[\cite{AD}]\label{lem2.4A}
Let $(T,X)$ be a minimal semiflow. If $(T,X)$ admits an invariant measure, then $(T,X)$ is surjective.
\end{lem}

\begin{lem}[{cf.~\cite[Lemma~7.6]{Aus} for $T$ a group}]\label{lem2.5A}
Let $\pi\colon(T,X)\rightarrow(T,Y)$ be an epimorphism where $(T,X)$ contains a dense set of minimal points and $(T,Y)$ is minimal invertible, and let $\varepsilon\in\mathscr{U}_X$. Then $(\pi\times\pi)(T^{-1}\varepsilon)$ belongs to $\mathscr{U}_Y$.
\end{lem}

\begin{proof}
Let $V\in\mathfrak{U}(X)$ such that $V\times V\subset\varepsilon$ and take a minimal point $x$ of $X$ with $x\in V$. Since $(T,Y)$ is minimal, hence $\pi(\textrm{cls}_XTx)=Y$. Thus there is some $\tau\in T$ such that $\textrm{Int}_Y\pi(\tau^{-1}V)\not=\emptyset$. Let $\eta=T^{-1}(\textrm{Int}_Y\pi(\tau^{-1}V)\times \textrm{Int}_Y\pi(\tau^{-1}V))$. Then $\eta$ is a non-empty $T^{-1}$-invariant open set which meets $\varDelta_Y$. Since $\varDelta_Y$ is minimal for $(T,Y\times Y)$ and $(Y\times Y)\setminus\eta$ is $T$-invariant, it follows that $\varDelta_Y\subset\eta$ and so $\eta\in\mathscr{U}_Y$. Then by
$(\pi\times\pi)(T^{-1}\varepsilon)\supseteq T^{-1}(\pi(\tau^{-1}V)\times\pi(\tau^{-1}V))\supseteq\eta$,
we can conclude that $(\pi\times\pi)(T^{-1}\varepsilon)\in\mathscr{U}_Y$. The proof is complete.
\end{proof}

The following result is useful, which generalizes \cite[Corollary to Theorem~8.1]{F63} and \cite[Proposition~2.3]{V70} that are in minimal distal flows with compact metric phase spaces by different approaches.

\begin{cor}\label{cor2.6A}
Let $\pi\colon(T,X)\rightarrow(T,Y)$ be an epimorphism of semiflows where $(T,X)$ has a dense set of minimal points and $(T,Y)$ is minimal invertible. Assume $\{(y_n,y_n^\prime)\}$ is a net in $Y\times Y$ such that $\lim_n(y_n,y_n^\prime)\in\varDelta_Y$. Then there are nets $\{(x_i,x_i^\prime)\}$ in $X\times X$ and $\{t_i\}$ in $T$ such that $\lim_it_i(x_i,x_i^\prime)\in\varDelta_X$ and that $\{(\pi(x_i),\pi(x_i^\prime))\}$ is a subnet of $\{(y_n,y_n^\prime)\}$.
\end{cor}

\begin{proof}
Given any $\varepsilon\in\mathscr{U}_X$, since $(\pi\times\pi)(T^{-1}\varepsilon)\in\mathscr{U}_Y$ by Lemma~\ref{lem2.5A} and $\lim_n(y_n,y_n^\prime)\in\varDelta_Y$, then we can take some $n_\varepsilon$ with $(y_{n_\varepsilon},y_{n_\varepsilon}^\prime)\in\{(y_n,y_n^\prime)\}$ such that $(y_{n_\varepsilon},y_{n_\varepsilon}^\prime)\in(\pi\times\pi)(T^{-1}\varepsilon)$. Let $t_\varepsilon\in T$ and $x_{n_\varepsilon},x_{n_\varepsilon}^\prime\in X$ such that $t_\varepsilon(x_{n_\varepsilon},x_{n_\varepsilon}^\prime)\in\varepsilon$ and $\pi(x_{n_\varepsilon})=y_{n_\varepsilon}$ and $\pi(x_{n_\varepsilon}^\prime)=y_{n_\varepsilon}^\prime$. This completes the proof of Corollary~\ref{cor2.6A}.
\end{proof}

Corollary~\ref{cor2.6A} is very useful for our later Theorem~\ref{thm2.7A}. Moreover, it should be noted that we are not permitted to use $(\pi\times\pi)(T\varepsilon)\in\mathscr{U}_Y$ instead of $(\pi\times\pi)(T^{-1}\varepsilon)\in\mathscr{U}_Y$, for $T\varepsilon$ is not necessarily open in our semiflow context.

\subsection{Regional proximity relations of extensions and factors}
Now we will be concerned with the relationship of $Q(X)$ with the same relation of its factors. The point of Theorem~\ref{thm2.7A} below is that $(T,X)$ need not be invertible.

\begin{thm}\label{thm2.7A}
Let $\pi\colon(T,X)\rightarrow(T,Y)$ be an epimorphism, where $(T,X)$ is surjective with a dense set of minimal points and $(T,Y)$ is invertible minimal. Then $(\pi\times\pi)Q(X)=Q(Y)$.
\end{thm}

\begin{proof}
Clearly $(\pi\times\pi)Q(X)\subseteq Q(Y)$. For the other direction inclusion, let $(y,y^\prime)\in Q(Y)$. Then there are nets $\{(y_n,y_n^\prime)\}$ in $Y\times Y$ and $\{\tau_n\}$ in $T$ with $(y_n,y_n^\prime)\to(y,y^\prime)$ and $\lim_n\tau_n(y_n,y_n^\prime)\in\varDelta_Y$. By Corollary~\ref{cor2.6A}, choose $(x_n,x_n^\prime)\in X\times X$ and $t_n\in T$ (passing to a subnet if necessary) with
$\lim_nt_n(x_n,x_n^\prime)\in\varDelta_X$ and $\pi(x_n)=\tau_ny_n,\ \pi(x_n^\prime)=\tau_n y_n^\prime$.
Pick $\tilde{x}_n,\tilde{x}_n^\prime\in X$ with $\tau_n(\tilde{x}_n,\tilde{x}_n^\prime)=(x_n,x_n^\prime)$. Then
$\lim t_n\tau_n(\tilde{x}_n,\tilde{x}_n^\prime)\in\varDelta_X$ and $\pi(\tilde{x}_n)=y_n,\ \pi(\tilde{x}_n^\prime)=y_n^\prime$.
By taking a subnet of $\{(\tilde{x}_n,\tilde{x}_n^\prime)\}$ if necessary, let $\tilde{x}_n\to x$ and $\tilde{x}_n^\prime\to x^\prime$. Thus $(x,x^\prime)\in Q(X)$ with $\pi(x)=y$ and $\pi(x^\prime)=y^\prime$. Thus $Q(Y)\subseteq(\pi\times\pi)Q(X)$. This proves Theorem~\ref{thm2.7A}.
\end{proof}

\begin{cor}[{cf.~\cite[Lemma~12]{EG} for $T$ a group}]\label{cor2.8A}
Let $\pi\colon(T,X)\rightarrow(T,Y)$ be an epimorphism such that $(T,X)$ is surjective with dense minimal points and (T,Y) is minimal invertible. Then $(T,Y)$ is equicontinuous if and only if $(\pi\times\pi)Q(X)=\varDelta_Y$.
\end{cor}

\begin{proof}
The ``only if'' follows from Lemma~\ref{lem2.3} and Theorem~\ref{thm2.7A}. Conversely, by Theorem~\ref{thm2.7A}, $Q(Y)=\varDelta_Y$ and so the ``if'' follows from Lemma~\ref{lem2.3}.
\end{proof}

Now applying Theorem~\ref{thm2.7A} with $\pi\colon(T,X)\rightarrow(T,X_\textit{eq})$, we can easily obtain the following result, which is very useful in surjective semiflows and which partially generalizes \cite[Proposition~4.20]{E69}.

\begin{cor}\label{cor2.9A}
Let $(T,X)$ be a surjective minimal semiflow. Then
$S_{\textit{eq}}(X)$ is the smallest closed invariant equivalence relation containing $Q(X)$.
\end{cor}

\begin{proof}
Since $(T,X_\textit{eq})$ is equicontinuous surjective, Lemma~\ref{lem2.2} follows that $(T,X_\textit{eq})$ is minimal invertible, where $X_\textit{eq}=X/S_\textit{eq}$. Then by Lemma~\ref{lem2.3}, $(\pi\times\pi)Q(X)\subseteq Q(X_\textit{eq})=\varDelta_{X_\textit{eq}}$ so that $Q(X)\subseteq S_{\textit{eq}}$.
To show $S_\textit{eq}(X)$ is the smallest one, let $R$ be a closed invariant equivalence relation on $X$ with $Q(X)\subseteq R$, $X\xrightarrow{\pi}X/R$ the canonical map and $E(X)\xrightarrow{\pi_*} E(X/R)$ the standard semigroup homomorphism between the Ellis enveloping semigroups induced by $\pi$. If $u^2=u\in E(X/R)$ then $\pi_*^{-1}(u)$ is a closed subsemigroup of $E(X)$. Hence there exists $v^2=v\in E(X)$ with $\pi_*(v)=u$. Since
$(x,v(x))\in P(X)\subseteq Q(X)\subseteq R$,
then $\pi(x)=\pi(v(x))=u\pi(x)$ for all $x\in X$. Thus $u=\textit{id}_{X/R}$ the identity of $X/R$. This implies that $(T,X/R)$ is distal and so invertible (cf.~Lemma~\ref{lem2.2}). Then by Theorem~\ref{thm2.7A} and $Q(X)\subseteq R$, $\varDelta_{X/R}=(\pi\times \pi)Q(X)=Q(X/R)$. Thus $(T,X/R)$ is equicontinuous by Lemma~\ref{lem2.3}. Whence $S_{\textit{eq}}\subseteq R$. This proves Corollary~\ref{cor2.9A}.
\end{proof}

We note here that when $T$ is a group, the proof of Corollary~\ref{cor2.9A} may be simplified much as follows:

\begin{proof}[\textbf{Another proof of Corollary~\ref{cor2.9A} for $T$ a group}]
Since $\pi\colon X\rightarrow X_\textit{eq}$ is an epimorphism of minimal flows, by $(\pi\times\pi)Q(X)\subseteq Q(X_\textit{eq})=\varDelta$ it follows that $Q(X)\subseteq S_\textit{eq}$. On the other hand, if $R$ is an invariant closed equivalence relation on $X$ with $Q(X)\subseteq R$, then by Theorem~\ref{thm2.7A} we can conclude that $Q(X/R)=\varDelta$ so $(T,X/R)$ is equicontinuous. Thus $S_\textit{eq}\subseteq R$. The proof is complete.
\end{proof}

Recall that Lemma~\ref{lem2.3} is proved in \cite[pp.~46-47]{DX} by using Lemma~\ref{lem2.2}, a.a. points and Veech's relation $V(X)$ in (1) of Definition~\ref{def1.13}. In fact, we can simply reprove it by only using Lemma~\ref{lem2.2} and Corollary~\ref{cor2.9A}.

\begin{proof}[\textbf{Another proof of Lemma~\ref{lem2.3}}]
Let $(T,X)$ be equicontinuous surjective; then $\varDelta_X=P(X)=Q(X)$ by Lemma~\ref{lem2.2}. Conversely, assume $Q(X)=\varDelta_X$; then $(T,X)$ is invertible and pointwise minimal. By Corollary~\ref{cor2.9A}, it follows that $S_\textit{eq}=\varDelta_X$ so $(T,X)$ is equicontinuous surjective.
\end{proof}

\begin{thm}[{cf.~\cite[Theorem~9.12]{Fur} for $T=\mathbb{Z}$}]\label{thm2.10}
Let $(T,X)$ be a minimal weak-mixing flow with $X$ non-trivial. Then no point of $X$ is distal.
\end{thm}

\begin{proof}
Since $(T,X)$ is weak-mixing, $Q(X)=X\times X$ and so $S_\textit{eq}=X\times X$ by Corollary~\ref{cor2.9A}. This shows that $(T,X)$ has no non-trivial equicontinuous factor. However, if $(T,X)$ had a distal point it were point-distal. So by Veech's Structure Theorem (cf.~\cite[Theorem~6.1]{V70} for $X$ metrizable and \cite[Corollary~2.2 and Lemma~2.3]{MN} for general $X$), $(T,X)$ would have a non-trivial equicontinuous factor. Thus $(T,X)$ has no distal point. The proof is complete.
\end{proof}
In addition, when $(T,X)$ is an invertible semiflow, Corollary~\ref{cor1.11} can be proved by using Corollary~\ref{cor2.9A} and Veech's Structure Theorem (\cite[Corollary~4.6]{AD}) as follows:

\begin{proof}[\textbf{Another proof of Corollary~\ref{cor1.11} when $(T,X)$ invertible}]
First $S_\textit{eq}=X\times X$ by Corollary~\ref{cor2.9A} and so $(T,X)$ has no non-trivial equicontinuous factor. However, if $(T,X)$ has a distal point with $T$ amenable, then by \cite[Corollary~4.6]{AD} it follows that $(T,X)$ has a non-trivial equicontinuous factor. Thus $(T,X)$ has no distal point and the proof is complete.
\end{proof}
\subsection{Quasi-regular invariant measure}
Since $X$ is not necessarily metrizable, hence a Borel probability measure does not need be regular. However, next we will show that $(T,X)$ admits an invariant Borel probability measure if and only if it admits an invariant quasi-regular (regular) Borel probability measure. Let $C_c(X)$ be the space of continuous real-valued functions with compact support.
\begin{quote}
\textbf{Riesz-Markov theorem}.~ {\it Let $X$ be a locally compact $\textrm{T}_2$ space and $\I$ a positive linear functional on $C_c(X)$. Then there is a Borel measure $\mu$ on $X$ such that
$$
\I(f)=\int_Xfd\mu\quad \forall f\in C_c(X).
$$
The measure $\mu$ may be taken to be quasi-regular. In this case, it is then unique.}
\end{quote}

\begin{lem}\label{lem2.11}
$(T,X)$ admits an invariant measure if and only if it admits an invariant quasi-regular Borel probability measure.
\end{lem}

\begin{proof}
The ``if'' part is trivial. So we now assume $(T,X)$ admits an invariant Borel probability measure $\nu$. Then by the Riesz-Markov theorem, we can obtain a positive linear functional
$$
\I(f)=\int_Xfd\nu\quad \forall f\in C(X).
$$
Since $\nu$ is $T$-invariant, $\I$ is also $T$-invariant in the sense that $\I(f)=\I(ft)$ for all $f\in C(X)$ and $t\in T$. Further by the Riesz-Markov theorem again, we can find a unique quasi-regular Borel probability measure $\mu$ such that
$$
\I(f)=\int_Xfd\mu\quad \forall f\in C(X).
$$
Since $\I$ is $T$-invariant, so $\mu$ is also invariant. Therefore the ``only if'' part holds.
\end{proof}

In view of Lemma~\ref{lem2.11}, we will identify an invariant Borel probability measure with an invariant quasi-regular Borel probability measure in our later arguments if no confusion arises.

\subsection{Furstenberg's structure theorem}\label{sec2.5}
Let $T$ be any discrete semigroup with identity $e$ and let $\theta\ge1$ be some ordinal. Following~\cite{F63},
a \textit{projective system} of minimal semiflows with phase semigroup $T$ is a collection of minimal semiflows $(T,X_\lambda)$ on compact $\textrm{T}_2$ spaces $X_\lambda$ indexed by ordinal numbers $\lambda\le\theta$, and a family of epimorphisms, $\pi_{\lambda,\nu}\colon(T,X_\lambda)\rightarrow(T,X_\nu)$, for $0\le\nu<\lambda\le\theta$, satisfying:
\begin{enumerate}
\item[(1)] If $0\le\nu<\lambda<\eta\le\theta$, then $\pi_{\eta,\nu}=\pi_{\lambda,\nu}\circ\pi_{\eta,\lambda}$.
\item[(2)] If $\mu\le\theta$ is a limit ordinal, then $X_\mu$ is the minimal subset of the Cartesian product semiflows $\left(T,\times_{\lambda<\mu}X_\lambda\right)$ consisting of all $x=(x_\lambda)_{\lambda<\mu}\in\times_{\lambda<\mu}X_\lambda$ with $x_\nu=\pi_{\lambda,\nu}(x_\lambda)$ for all $\nu<\lambda<\mu$ and then for $\lambda<\mu$,
    $\pi_{\mu,\lambda}\colon X_\mu\rightarrow X_\lambda$ is just the projection map. In this case, we say that $(T,X_\mu)$ is the \textit{projective limit} of the family of minimal semiflows $\{(T,X_\lambda)\,|\,\lambda<\mu\}$.
\end{enumerate}

Let $\pi\colon(T,X)\rightarrow(T,Y)$ be an epimorphism of two semiflows.
Then $\pi$ is called \textit{relatively equicontinuous} if given $\varepsilon\in\mathscr{U}_X$ there is a $\delta\in\mathscr{U}_X$ such that whenever $(x,x^\prime)\in\delta$ with $\pi(x)=\pi(x^\prime)$, then $(tx,tx^\prime)\in\varepsilon$ for all $t\in T$ (cf.~\cite{F63} and \cite[p.~95]{Aus}). In this case, $(T,X)$ is also call a \textit{relatively equicontinuous extension} of $(T,Y)$.

Now based on these definitions, we are ready to state the Furstenberg structure theorem for minimal distal semiflows as follows:

\begin{quote}
\textbf{Furstenberg's Structure Theorem} (cf.~\cite[Theorem~5.14]{AD}).~ {\it Let $\pi\colon (T,X)\rightarrow(T,Y)$ be an epimorphism between distal minimal semiflows. Then there is a projective system of minimal semiflows $\{(T,X_\lambda)\,|\,\lambda\le\theta\}$, for some ordinal $\theta\ge1$, with $X_\theta=X$, $X_0=Y$ such that if $0\le\lambda<\theta$, then $(T,X_{\lambda+1})\xrightarrow{\pi_{\lambda+1,\lambda}}(T,X_\lambda)$ is a relatively equicontinuous extension.}
\end{quote}

In fact, we will only need the special case that $(T,Y)$ is the semiflow with $Y$ a singleton space. See Lemma~\ref{lem4.1} below.
\section{McMahon pseudo-metric and $S_\textit{eq}=Q$}\label{sec3}
This section will be mainly devoted to proving Theorem~\ref{thm1.4} and considering another proximity relation of Veech.
The McMahon pseudo-metric $D_J$ and the induced equivalence relation $R_J$ based on an invariant closed subset $J$ of $X\times X$ are useful techniques for our aim here.

\subsection{McMahon pseudo-metric}\label{sec3.1}
In this subsection, let $(T,X)$ be a minimal semiflow; and suppose $(T,Y)$ is any semiflow, which admits an invariant (quasi-regular Borel probability) measure $\mu$.
If $J$ is a closed invariant subset of $X\times Y$ and $x\in X$, then the \textit{section} $J_x$ is defined by
$$J_x=\{y\in Y\,|\,(x,y)\in J\}.$$
Such $J$ is called a \textit{joining} of $(T,X)$ and $(T,Y)$ in some works if we additionally require that $\pi_Y(J)=Y$ where $\pi_Y\colon(x,y)\mapsto y$.

\begin{lem}[{cf.~\cite[Lemma~9.4]{Aus} for $T$ a group}]\label{lem3.1}
Let $J$ be a closed invariant subset of $X\times Y$. If $x,x^\prime\in X$, then $\mu(J_x)=\mu(J_{x^\prime})$.
\end{lem}

\begin{proof}
Let $\varepsilon>0$ and let $V$ be an open set in $Y$ with $J_x\subset V$ and $\mu(V)<\mu(J_x)+\varepsilon$. Let $\{t_n\}$ be a net in $T$ for which $t_nx^\prime\to x$. It is easy to see that $J_{t_nx^\prime}\subset V$ for $n\ge n_0$, for some $n_0$. Hence $\mu(J_{x^\prime})\le\mu(J_{t_nx^\prime})\le\mu(V)\le\mu(J_x)+\varepsilon$. Letting $\varepsilon\to 0$, we have $\mu(J_{x^\prime})\le\mu(J_x)$. By symmetry, $\mu(J_x)\le\mu(J_{x^\prime})$ and the lemma is thus proved.
\end{proof}

\begin{defn}\label{def3.2}
If $J$ is a closed invariant subset of $X\times Y$, we define the \textit{McMahon pseudo-metric} $D_J$ on $X$ by
$$
D_J(x,x^\prime)=\mu(J_x\vartriangle J_{x^\prime})\quad \forall x,x^\prime\in X.
$$
It is easy to check that $D_J$ is a pseudo-metric on $X$. Moreover, $D_J(x,x^\prime)=0$ if $J_x\subseteq J_{x^\prime}$ or if $J_x\supseteq J_{x^\prime}$ by Lemma~\ref{lem3.1}.
\end{defn}

\begin{lem}[{cf.~\cite[Lemma~9.5]{Aus} for $T$ a group}]\label{lem3.3}
The pseudo-metric $D_J$ on $X$ is continuous and satisfies $D_J(x,x^\prime)=D_J(tx, tx^\prime)$ for all $t\in T$ and any $x,x^\prime\in X$.
\end{lem}

\begin{proof}
Let $x,x^\prime\in X$ and $t\in T$.
First note that $tJ_x\subseteq J_{tx}$ and $\mu(tJ_x)=\mu(J_{tx})$ by Lemma~\ref{lem3.1}. So $tJ_x=J_{tx}\pmod0$; that is, $\mu(tJ_x\vartriangle J_{tx})=0$. Similarly, $tJ_{x^\prime}\subseteq J_{tx^\prime}$ and so $tJ_{x^\prime}=J_{tx^\prime}\pmod0$. Thus
$$t^{-1}(tJ_x)=J_x\pmod0\quad \textrm{and}\quad t^{-1}(tJ_{x^\prime})=J_{x^\prime}\pmod0.$$
This implies that $J_x\cap J_{x^\prime}= t^{-1}(tJ_x\cap tJ_{x^\prime})=t^{-1}(J_{tx}\cap J_{tx^\prime})\pmod0$. Thus by Lemma~\ref{lem3.1}, it follows that $D_J(x,x^\prime)=\mu(J_x)+\mu(J_{x^\prime})-\mu(J_x\cap J_{x^\prime})=\mu(J_{tx})+\mu(J_{tx^\prime})-\mu(J_{tx}\cap J_{tx^\prime})=D_J(tx,tx^\prime)$.

To show that $D_J$ is continuous, we first note that if $U$ is an open set in $Y$, with $J_x\subset U$, and $\mu(U-J_x)<\epsilon$, then $\mu(U-J_{x^\prime})<\epsilon$, for $x^\prime$ sufficiently close to $x$. Now, let $\{x_n\}$ be a net in $X$ with $x_n\to x$. Let $\epsilon>0$, $U$ open in $Y$ with $J_x\subset U$ and $\mu(U-J_x)<\epsilon$. Then if $n\ge n_0$, $J_{x_n}\subset U$ and then $\mu(J_x-J_{x_n})<\epsilon$, so $D_J(x,x_n)<2\epsilon$. Thus, if $x_n\to x$, then $D_J(x,x_n)\to 0$. It follows immediately that $D_J$ is continuous. In fact, let $(x_n,x_n^\prime)\to (x,x^\prime)$, then
$$
\lim_n|D_J(x,x^\prime)-D_J(x_n,x_n^\prime)|\le\lim_n|D_J(x,x_n)+D_J(x_n,x_n^\prime)+D_J(x_n^\prime,x^\prime)-D_J(x_n,x_n^\prime)|=0.
$$
The proof of Lemma~\ref{lem3.3} is therefore completed.
\end{proof}

\begin{sn}
Let $R_J$ be the equivalence relation on $X$ defined by the pseudo-metric $D_J$ as follows: $(x,x^\prime)\in R_J$ if and only if $D_J(x,x^\prime)=0$.
\end{sn}

\begin{lem}[{cf.~\cite[Lemma~9.6]{Aus} for $T$ a group}]\label{lem3.5}
$R_J$ is a closed $T$-invariant equivalence relation on $X$ which contains $Q(X)$.
\end{lem}

\begin{proof}
Lemma~\ref{lem3.3} implies that $R_J$ is a closed invariant equivalence relation. Now if $(x,x^\prime)$ belongs to $Q(X)$, let $x_n\to x, x_n^\prime\to x^\prime$ and $t_n(x_n,x_n^\prime)\to(z,z)$ for some $z\in X$, as in Definition~\ref{def1.3}. Then $D_J(x_n,x_n^\prime)\to D_J(x,x^\prime)$ and $D_J(x_n,x_n^\prime)=D_J(t_nx_n,t_nx_n^\prime)\to D_J(z,z)=0$. So $D_J(x,x^\prime)=0$ and thus $(x,x^\prime)\in R_J$. This shows $Q(X)\subseteq R_J$.
\end{proof}

\begin{thm}\label{thm3.6}
If $J$ is a closed invariant subset of $(T,X\times Y)$ and if $(T,X)$ is surjective. Then $S_{\textit{eq}}(X)\subseteq R_J$.
\end{thm}

\begin{proof}
This follows immediately from Lemma~\ref{lem3.5} and Corollary~\ref{cor2.9A}.
%
\end{proof}

It should be noticed that in Theorem~\ref{thm3.6}, we do not require that $\mu$ is such that $\mu(V)>0$ for all non-empty open subset $V$ of $Y$. For example, if $\mu$ is exactly concentrated on a fixed point, then $R_J=X\times X$ and so $(T,X/R_J)$ is trivial.

\subsection{$S_\textit{eq}(X)=Q(X)$}\label{sec3.2}
Now we will specialize to the case $(T,Y)=(T,X)$, which admits an invariant measure $\mu$. There is no loss of generality in assuming $\mu$ is quasi-regular by Lemma~\ref{lem2.11}.

\begin{thm1.4}
Let $(T,X)$ be a minimal semiflow, which admits an invariant quasi-regular measure $\mu$ on $X$. Then $S_\textit{eq}(X)=Q(X)$.
\end{thm1.4}

\begin{proof}
First by Lemma~\ref{lem2.4A}, $(T,X)$ is surjective.
Thus $Q(X)\subseteq S_{\textit{eq}}(X)$ by Corollary~\ref{cor2.9A}. To prove Theorem~\ref{thm1.4}, it is sufficient to show that $S_{\textit{eq}}(X)\subseteq Q(X)$.

Let $(x,y)\in S_{\textit{eq}}(X)$, and let $V$ be a neighborhood of $x$. Consider the closed invariant subset $J$ of $X\times X$ defined by
$$
J=\overline{{\bigcup}_{x^\prime\in V}T(y,x^\prime)}=\overline{{\bigcup}_{x^\prime\in V}\overline{T(y,x^\prime)}}.
$$
Hence
$$
J=\mathrm{cls}_{X\times X}\left\{(z,w)\,|\,\exists\,x^\prime\in V\textrm{ and }\exists\,\{t_n\}\textrm{ in }T\textit{ s.t. }t_n(y,x^\prime)\to(z,w)\right\}.
$$
Then $V\subset J_y$ (by taking $\{t_n\}=\{e\}$). Now since $(x,y)\in S_{\textit{eq}}(X)$, so $(x,y)\in R_J$ by Theorem~\ref{thm3.6}, and it follows that $V\subseteq J_x\pmod 0$ and so $V\subseteq J_x$ for $\textrm{supp}\,(\mu)=X$.

Summarizing, if $(x,y)\in S_{\textit{eq}}(X)$, and $V$ is a neighborhood of $x$, then there is an $x^\prime\in V$ and a $\tau\in T$ such that $\tau x^\prime$ and $\tau y$ are in $V$. Since $V$ is arbitrary, it follows that $(x,y)\in Q(X)$. Moreover, $(x,x)\in J$ by $V\subseteq J_x$.

The proof of Theorem~\ref{thm1.4} is thus completed.
\end{proof}

Note that the proof just given shows that in the definition of $Q(X)$ one can take one of the nets in $X$ to be constant. Precisely, we have:

\begin{lem}\label{lem3.7}
Under the hypotheses of Theorem~\ref{thm1.4}, the following conditions are pairwise equivalent:
\begin{enumerate}
\item[$(1)$] $(x,y)\in Q(X)$.
\item[$(2)$] There are nets $\{x_n\}$ in $X$ and $\{t_n\}$ in $T$ with $x_n\to x$, $t_nx_n\to x$ and $t_ny\to x$.
\item[$(3)$] There are nets $\{x_n\}$ in $X$ and $\{s_n\}$ in $T$ such that $x_n\to x, s_nx_n\to y$ and $s_ny\to y$.
\end{enumerate}
\end{lem}

\begin{note}
(1) $\Leftrightarrow$ (2) in flows is due to McMahon~\cite{McM}.
\end{note}

\begin{proof}
``(1) $\Leftrightarrow$ (2)'' and ``(3) $\Rightarrow$'' (1) are obvious. For ``(1) $\Rightarrow$ (3)'', as in the proof of Theorem~\ref{thm1.4}, since $(x,x)\in J$ and $J$ is closed invariant, so by minimality of $(T,X)$, it follows that $(y,y)\in J$. Thus, we can find nets $\{x_n\}$ and $\{s_n\}$ such that $x_n\to x$, $s_nx_n\to y$ and $s_ny\to y$. This proves Lemma~\ref{lem3.7}.
\end{proof}

The (1) $\Leftrightarrow$ (2) of Lemma~\ref{lem3.7} implies the following (cf.~\cite[Corollary~9.10]{Aus} for $T$ a group).

\begin{cor}\label{cor3.8}
Let $(T,X)$ be a minimal semiflow admitting an invariant measure. Then,
\begin{gather*}
Q[y]={\bigcap}_{\alpha\in\mathscr{U}_X}\overline{{\bigcup}_{t\in T}(t^{-1}\alpha)[y]},
\end{gather*}
for all $y\in X$.
\end{cor}
\begin{proof}
First, if $x\in\bigcap_{\alpha\in\mathscr{U}_X}\overline{\bigcup_{t\in T}(t^{-1}\alpha)[y]}$, then for every $\alpha\in\mathscr{U}_X$, there are $x_\alpha\in\alpha[x]$ and $t_\alpha\in T$ with $(t_\alpha y,t_\alpha x_\alpha)\in\alpha$. This shows that $x\in Q[y]$. Conversely, if $x\in Q[y]$, then by Lemma~\ref{lem3.7}, $x\in\overline{\bigcup_{t\in T}(t^{-1}\alpha)[y]}$ for all $\alpha\in\mathscr{U}_X$. Thus $x\in\bigcap_{\alpha\in\mathscr{U}_X}\overline{\bigcup_{t\in T}(t^{-1}\alpha)[y]}$.
\end{proof}

The (1) $\Leftrightarrow$ (3) of Lemma~\ref{lem3.7} is very useful for proving Theorem~\ref{thm1.15} in $\S\ref{sec5}$. The relation in (3) of Lemma~\ref{lem3.7} was first introduced and studied by Veech in \cite{V77}. See $\S\ref{sec3.3}$ for the details.

\subsection{Another proximity relation and the proof of Theorem~\ref{thm1.15}(1)}\label{sec3.3}
Following Veech~\cite[p.~806 and p.~819]{V77}, we introduce the following notation.

\begin{defn}
Let $(T,X)$ be a minimal semiflow.
\begin{enumerate}
\item We say $(T,X)$ satisfies the \textit{Bronstein condition} if $(T,X\times X)$ contains a dense set of minimal points.

\item Given $x\in X$, define $U[x]$ to be the set of $z\in X$ for which there exist nets $t_n\in T$ and $z_n\in X$ such that $z_n\to z$, $t_nz_n\to x$ and $t_nx\to x$.
\end{enumerate}
\end{defn}

It is clear that $P[x]\subseteq U[x]\subseteq Q[x]$ for all $x\in X$. In \cite{V77} Veech proved the following theorem:

\begin{quote}
\textbf{Theorem} (cf.~\cite[Theorem~2.7.6]{V77}).~ {\it Let $(T,X)$ be a minimal flow satisfying the Bronstein condition. Then $S_\textit{eq}[x]=U[x]=Q[x]$ for all $x\in X$.}
\end{quote}

Therefore, it holds that

\begin{quote}
\textbf{Corollary} (cf.~\cite[Theorem~2.7.5]{V77}).~ {\it If $(T,X)$ is a minimal distal flow, then we have $S_\textit{eq}[x]=U[x]=Q[x]$ for all $x\in X$.}
\end{quote}

With an invariant measure instead of the Bronstein condition, by using Lemma~\ref{lem3.7} and Theorem~\ref{thm1.4} we can easily obtain the following.

\begin{thm}\label{thm3.10}
Let $(T,X)$ be a minimal semiflow, which admits an invariant measure. Then $U[x]=Q[x]=S_\textit{eq}[x]$ for all $x\in X$.
\end{thm}

Since every distal semiflow always has an invariant measure by Furstenberg's theorem, we can easily obtain the following corollary to Theorem~\ref{thm3.10}.

\begin{cor}\label{cor3.11}
If $(T,X)$ be a minimal distal semiflow, then $S_\textit{eq}[x]=U[x]=Q[x]$ for all $x\in X$.
\end{cor}

In fact, we have the following, where $D[x]$ is as in (4) of Definition~\ref{def1.13}.

\begin{lem}[{cf.~\cite[Theorem~5.6]{Dai}}]\label{lem3.12}
If $(T,X)$ is a minimal bi-continuous semiflow, then we have $D[x]=U[x]$ for all $x\in X$.
\end{lem}

\begin{cor}\label{cor3.13}
If $(T,X)$ be a minimal bi-continuous semiflow, then $D[x]\subseteq Q[x]$ for all $x\in X$.
\end{cor}

\begin{proof}[Proof of Theorem~\ref{thm1.15}(1)] 
This follows at once from Theorem~\ref{thm3.10} and Lemma~\ref{lem3.12}.
\end{proof}
\section{Weak-mixing minimal semiflows}\label{sec4}
In this section, we will characterize a minimal weak-mixing semiflow by using the McMahon pseudo-metric $D_J$ introduced in $\S\ref{sec3.1}$. Moreover, we will consider the chaotic dynamics of minimal weak-mixing semiflow with amenable phase semigroup and with metrizable phase space.

\subsection{Characterizations of minimal weak-mixing semiflows}\label{sec4.1}
Let $(T,X)$ be a surjective semiflow. Recall that $\mathfrak{U}(X)$ is the collection of non-empty open subsets of $X$.
Then it is clear that the weak-mixing is ``highly non-equicontinuous''.
In fact, if $(T,X)$ is weak-mixing, then $Q(X)=X\times X$; and then further its factor $(T,X_{\textit{eq}})$ is trivial by Corollary~\ref{cor2.9A}.
\newpage

\begin{lem}\label{lem4.1}
Let $(T,X)$ be a minimal surjective semiflow. Then the following two conditions are equivalent:
\begin{enumerate}
\item[$(1)$] $(T,X)$ has no non-trivial distal factor.

\item[$(2)$] $(T,X)$ has no non-trivial equicontinuous factor; that is, $S_\textit{eq}(X)=X\times X$.
\end{enumerate}
\end{lem}

\begin{proof}
$(1)\Rightarrow(2)$ by Lemma~\ref{lem2.2}. And $(2)\Rightarrow (1)$ follows easily from Furstenberg's structure theorem stated in $\S\ref{sec2.5}$.
\end{proof}

Therefore, any minimal surjective semiflow is ``highly non-equicontinuous'' if and only if it is ``highly non-distal.''

\begin{defn}[{cf.~\cite{Pel}}]\label{def4.2}
Let $(T,X)$ and $(T,Y)$ be two semiflows with compact $\textrm{T}_2$ phase spaces and with the same phase semigroup $T$.
We will say $(T,X)$ is \textit{weakly disjoint} from $(T,Y)$, denoted $X\perp^wY$, if $(T,X\times Y)$ is a T.T. semiflow.
\end{defn}

It should be noted here that since our phase spaces need not be metrizable, Definition~\ref{def4.2} is not equivalent to requiring that $(T,X\times Y)$ is point-transitive even for $T$ in groups as in \cite[p.~150]{Aus}.

Let $(f,X)$ be a cascade where $f$ is a homeomorphism on a compact metric space, and assume $(f,X)$ has no non-trivial equicontinuous factor. Then $(f,X\times X)$ is T.T. (\cite{KR}). Next we will consider an open question of Furstenberg.

Let $T$ be a semigroup; then the class of minimal semiflow with phase semigroup $T$ will be denoted by $\texttt{SF}_{\min}$, and we write
\begin{gather*}
\texttt{SF}_{tt}=\{(T,X)\,|\,(T,X)\textrm{ is T.T.}\}\quad \textrm{and}\quad\texttt{SF}_{wm}=\{(T,X)\,|\, (T,X\times X)\textrm{ is T.T.}\}.
\end{gather*}
Note that all the phase spaces are compact $\textrm{T}_2$ for our semiflows. Moreover, $\texttt{SF}_{\min}\subset\texttt{SF}_{tt}$ for all semigroup $T$.

Furstenberg's \cite[Proposition~II.11]{F67} asserts that $\texttt{SF}_{wm}\times\texttt{SF}_{\min}\subset\texttt{SF}_{tt}$ if $T=\mathbb{Z}_+$. In view of this, he further asked the following problem:

\begin{quote}
\textit{Is it true for} $T=\mathbb{Z}_+$ \textit{that} $\texttt{SF}_{wm}\times\texttt{SF}_{tt}\subset\texttt{SF}_{tt}$\textit{?} (See \cite[Problem~F]{F67}.)
\end{quote}
This is false in general; however we will consider the following question:
\begin{quote}
{\it $\left(\texttt{SF}_{wm}\cap\texttt{SF}_{\min}\right)\times\left(\texttt{SF}_{tt}\cap \textit{\textbf{?}}\right)\subset\texttt{SF}_{tt}$}
\end{quote}
As to this, there is an important theorem due to Veech:
\begin{quote}
\textbf{Theorem~J} (Veech~\cite[Theorem~2.1.6]{V77}).~{\it Let $T$ be a group. Let $(T,X)$ be ``incontractible'' minimal and assume $(T,X)$ has no nontrivial equicontinuous factor. Let $(T,Y)$ be T.T. having a dense set of a.p. points. Then $(T,X\times Y)$ is a T.T. flow.}
\end{quote}

Next we will introduce a class of semiflows which are weaker than minimal semiflows with amenable phase semigroups.

\begin{defn}\label{def4.3}
Let $T$ be a discrete semigroup and let $X$ vary in the set of compact $\textrm{T}_2$ spaces.
\begin{enumerate}
\item By the \textit{ergodic center} of $(T,X)$, denoted by $\mathscr{C}_{erg}(T,X)$, we mean the smallest closed invariant subset of $X$ of $\mu$-measure $1$, for all invariant measure $\mu$ of $(T,X)$. If $(T,X)$ has no invariant measure, then we shall say $\mathscr{C}_{erg}(T,X)=\emptyset$.

\item $(T,X)$ is called an \textit{E-semiflow}, denoted $(T,X)\in\texttt{SF}_{e}$, if $(T,X)\in\texttt{SF}_{tt}$ and $\mathscr{C}_{erg}(T,X)=X$.
Equivalently, $(T,X)\in\texttt{SF}_{e}$ if and only if it is T.T. with full ergodic center.
\end{enumerate}
\end{defn}

The following lemma is a simple observation and so we will omit its proof details.

\begin{lem}\label{lem4.4}
Let $T$ be an amenable semigroup. If $(T,X)$ is T.T. such that the set of minimal points is dense in $X$, then $(T,X)\in\texttt{SF}_{e}$.
\end{lem}

$\texttt{SF}_{e}$ is an extension of the class of \textit{E}-systems of Glasner and Weiss~\cite{GW1, GW2}. If $(T,X)\in\texttt{SF}_{e}$ with $T$ a countable discrete semigroup, then $N_T(U,V)$ is syndetic in $T$ for all $U,V\in \mathfrak{U}(X)$.

\begin{lem}\label{lem4.5}
If $(T,X)\in\texttt{SF}_{e}$ and $U\in\mathfrak{U}(X)$, then there exists an invariant measure $\mu$ of $(T,X)$ such that $\mu(U)>0$. Hence every $(T,X)$ in $\texttt{SF}_{e}$ is surjective.
\end{lem}

\begin{proof}
First we note that $\mathscr{C}_{erg}(T,X)$ is just equal to the closure of the union of the supports of all invariant measures of $(T,X)$. Then there is some invariant measure $\mu$ such that $\textrm{supp}\,(\mu)\cap U\not=\emptyset$ so that $\mu(U)>0$. Finally, let $t\in T$. Since $X\setminus tX$ is open and $tX$ is Borel, so if $tX\not=X$ we have $1>\mu(tX)=\mu(t^{-1}tX)=\mu(X)=1$ a contradiction. Thus $(T,X)$ is surjective. The proof is complete.
\end{proof}

\begin{thm}\label{thm4.6}
Let $(T,X)\in\texttt{SF}_{\min}$ be surjective. Then $(T,X)$ has no non-trivial distal factor if and only if $X\perp^wY$ for all $(T,Y)\in\texttt{SF}_{e}$.
\end{thm}

\begin{proof}
The ``only if'' part: By Lemma~\ref{lem4.1} we may assume $(T,X)$ has no non-trivial equicontinuous factor $X_\textit{eq}$. Let $(T,Y)$ be TT with $\mathscr{C}_{erg}(T,Y)=Y$.
If $(T,X\times Y)\not\in\texttt{SF}_{tt}$, then there would be an invariant closed subset $J$ of $X\times Y$ such that $\textrm{Int}_{X\times Y}J\not=\emptyset$ and that $J\not=X\times Y$. So there could be found a point $x\in X$ and an $U\in\mathfrak{U}(Y)$ such that $U\subset J_x$. Moreover, $J_x\not=Y$. Let $V=Y-J_x$. Then $V\not=\emptyset$. Since $(T,Y)$ is TT, so there is $\tau\in T$ such that $\tau U\cap V\not=\emptyset$.

Since $U\cap\tau^{-1}V\in\mathfrak{U}(Y)$, we can find an invariant quasi-regular Borel probability $\mu$ on $Y$ such that $\mu(U\cap\tau^{-1}V)>0$.
Since $\tau U\cap V=\tau(U\cap\tau^{-1}V)$, hence $\mu(\tau J_x\cap V)>0$ and then $\mu(J_{\tau x}\cap V)>0$. Having let $D_J$ be the McMahon pseudo-metric as in Definition~\ref{def3.2} associated to $J$ and $\mu$, $D_J(x,\tau x)>0$. Let $R_J$ be the closed invariant equivalence relation defined by $D_J$ on $X$. Then $R_J\not=X\times X$, and by Theorem~\ref{thm3.6}, $(T,X/R_J)$ is a non-trivial equicontinuous factor of $(T,X)$. This contradiction shows that $J=X\times Y$ and therefore $(T,X\times Y)$ must be TT
and so $X\perp^wY$.

The ``if" part: Let $X\perp^wY$ for all $(T,Y)\in\texttt{SF}_{e}$. To be contrary, assume $X_\textit{eq}$ is not a singleton set. Then $(T,X_\textit{eq})$ is minimal distal by Lemma~\ref{lem2.2}. So $(T,X_\textit{eq})\in\texttt{SF}_{e}$ and thus $X\perp^wX_\textit{eq}$. Therefore $X_\textit{eq}\perp^wX_\textit{eq}$ and $(T,X_\textit{eq})$ is weak-mixing. This is a contradiction to Lemma~\ref{lem2.3}. Thus $X_\textit{eq}$ must be a singleton set.

Therefore the proof of Theorem~\ref{thm4.6} is completed.
\end{proof}

It should be noted that $(T,X)$ need not admit an invariant measure in the above Theorem~\ref{thm4.6}.

\begin{lem}\label{lem4.7}
Let $(T,X)\in\texttt{SF}_{\min}$ admits an invariant measure. Then it has no non-trivial equicontinuous factor if and only if it is weak-mixing.
\end{lem}

\begin{proof}
First $(T,X)$ is surjective by Lemma~\ref{lem2.4A}, and moreover, $(T,X)\in\texttt{SF}_{e}$. If $(T,X)$ has no non-trivial equicontinuous factor, then it is weak-mixing from Theorem~\ref{thm4.6}. Conversely, if $(T,X)$ is weak-mixing, then $Q(X)=X\times X$ so that $(T,X)$ has no non-trivial equicontinuous factor by Corollary~\ref{cor2.9A}. The proof is complete.
\end{proof}

The following Corollary~\ref{cor4.8} implies the first ``if and only if'' of Theorem~\ref{thm1.7}. In addition, if $T$ is a nilpotent group (then it is amenable), every minimal flow with phase group $T$ is incontractible; thus for this case the forgoing Theorem~J of Veech may follow from Corollary~\ref{cor4.8} and Theorem~\ref{thm4.12} below.

\begin{cor}\label{cor4.8}
Let $(T,X)\in\texttt{SF}_{\min}$ admits an invariant measure. Then $(T,X)$ is weak-mixing iff $X\perp^wY$ for all $(T,Y)\in\texttt{SF}_{e}$.
\end{cor}

\begin{cor}[{cf.~\cite[Corollary~11.8]{Aus} for $T$ in abelian groups}]\label{cor4.9}
Let $T$ be an amenable semigroup and $(T,X)\in\texttt{SF}_{\min}$. Then $(T,X)$ is weak-mixing iff $X\perp^wY$ for all $(T,Y)\in\texttt{SF}_{\min}$ iff $X\perp^wY$ for all $(T,Y)\in\texttt{SF}_{e}$.
\end{cor}

Comparing with \cite[Corollary~11.8]{Aus} here our phase spaces are not required to be a metric space and further `T.T.' does not imply `point-transitive' in our setting. Moreover, the minimality of $(T,X)$ is important for Corollary~\ref{cor4.9}. In fact, following \cite{AG,HY} there are non-minimal non-weak-mixing cascade $(T,X)$, which are weakly disjoint every \textit{E}-system.

\begin{cor}\label{cor4.10}
Let $(T,X)\in\texttt{SF}_{\min}$ be surjective. If $(T,X)$ is weak-mixing, then $X\perp^wY$ for all $(T,Y)\in\texttt{SF}_{e}$.
\end{cor}

\begin{proof}
By Corollary~\ref{cor2.9A}, $S_\textit{eq}(X)\supseteq Q(X)$. Then the statement follows from Theorem~\ref{thm4.6} and the fact that $Q(X)=X\times X$.
\end{proof}

Now Theorem~\ref{thm1.8} stated in $\S\ref{sec1.2}$ easily follows from the following more general theorem. Our proof below is of interest because there exists no Furstenberg's intersection lemma (\cite[Proposition~II.3]{F67} and \cite[Lemma~3.2]{DT}) in the literature for non-abelian phase semigroup $T$.

\begin{thm}\label{thm4.11}
Let $(T,X)$ be a minimal semiflow, which admits an invariant measure. Then the following conditions are pairwise equivalent:
\begin{enumerate}
\item[$(1)$] $(T,X)$ is weak-mixing.
\item[$(2)$] $(T,X^n)$ is T.T. for all $n\ge2$.
\item[$(3)$] $N_T(U,V)$ is discretely thick in $T$ for all $U,V\in\mathfrak{U}(X)$.
\item[$(4)$] Let $I$ be any set with $\textrm{card}\,I\ge2$ and $(T,X_i)=(T,X)$ for all $i\in I$; then $(T,\times_{i\in I}X_i)$ is weak-mixing.
\end{enumerate}
\end{thm}

\begin{proof}
$(1)\Rightarrow(2)$. We will proceed to show that $(T,X^n)$ is T.T. for all $n\ge2$ by induction on $n$.
First, by definition of weak-mixing, $(T,X\times X)$ is T.T. so the case of $n=2$ holds.
Now, assume the statement holds for all integer $n=k\ge2$. Let $Y=X^k$ and then $(T,Y)$ is a T.T. semiflow. Since $(T,X)$ is minimal with an invariant measure, then $(T,X)$ is surjective by Lemma~\ref{lem2.4A} and there is an invariant measure $\nu$ on $X$ with $\textrm{supp}\,(\nu)=X$. Define a Borel probability measure $\mu=\nu^k$ on $Y$ by $\mu(V_1\times\dotsm\times V_k)=\nu(V_1)\dotsm\nu(V_k)$ for all open subsets $V_1,\dotsc,V_k$ of $X$. Clearly, $\mu$ is $T$-invariant on $Y$ and moreover every open subset of $Y$ has positive $\mu$-measure. Then by Theorem~\ref{thm4.6}, $(T,X\times Y)$ must be T.T. and thus $(T,X^{k+1})$ is a T.T. semiflow.

$(2)\Rightarrow(3)$. Let $U,V\in\mathfrak{U}(X)$ and $K=\{k_1,\dotsc,k_n\}$ a finite subset of $T$. Since by Lemma~\ref{lem2.4A} each $k_i^{-1}V\not=\emptyset$, $U\times\dotsm\times U$ and $k_1^{-1}V\times\dotsm\times k_n^{-1}V$ both are non-empty open subsets of $X^n$, then there is some $t\in T$ such that $U\times\dotsm\times U\cap t^{-1}(k_1^{-1}V\times\dotsm\times k_n^{-1}V)\not=\emptyset$. Thus $Kt\subseteq N_T(U,V)$ and so $N_T(U,V)$ is discretely thick in $T$.

$(3)\Rightarrow(1)$. Let $U,V,U^\prime, V^\prime\in\mathfrak{U}(X)$. Then $N_T(U,V)$ is discretely thick in $T$; moreover, $N_T(U^\prime, V^\prime)$ is discretely syndetic in $T$ by \cite[Lemma~2.3]{DT}. So $N_T(U,V)\cap N_T(U^\prime, V^\prime)\not=\emptyset$. This shows that $N_T(U\times U^\prime, V\times V^\prime)\not=\emptyset$ and thus $(T,X\times X)$ is a T.T. semiflow.

$(2)\Leftrightarrow(4)$. $(4)\Rightarrow(2)$ is obvious; and $(2)\Rightarrow(4)$ follows from the definition of the product topology of $\times_{i\in I}X_i$.

The proof of Theorem~\ref{thm4.11} is thus completed.
\end{proof}

The following Theorem~\ref{thm4.12} implies that:
\begin{quote}
If $(T,X)$ is minimal with $T$ an amenable semigroup, then $(T,X)$ is weak-mixing if and only if it has no non-trivial distal factor.
\end{quote}
This would just answer a question of Petersen in~\cite[p.~280]{P70}. Moreover Theorem~\ref{thm4.12} is a semigroup version of Theorem~\ref{thm2.10} by different approaches.

\begin{thm}\label{thm4.12}
Let $(T,X)$ be a minimal semiflow, which admits of an invariant measure. Then, the following five statements are pairwise equivalent:
\begin{enumerate}
\item[$(1)$] $(T,X)$ is weak-mixing.
\item[$(2)$] $Q(X)=X\times X$.
\item[$(3)$] $S_{\textit{eq}}(T,X)=X\times X$.
\item[$(4)$] $(T,X)$ has no non-trivial distal factor.
\item[$(5)$] $(T,X)$ has no non-trivial equicontinuous factor.
\end{enumerate}
In particular, each of $(1)$ through $(5)$ implies that $(T,X)$ has no distal point.
\end{thm}

\begin{proof}
First of all, $(T,X)$ is surjective by Lemma~\ref{lem2.4A}. $(1)\Rightarrow(2)$ and $(3)\Rightarrow(5)$ both are trivial; $(2)\Rightarrow(3)$ is by Corollary~\ref{cor2.9A} and $(4)\Leftrightarrow(5)$ follows easily from Lemma~\ref{lem4.1}. Finally $(5)\Rightarrow(1)$ follows from Lemma~\ref{lem4.7} at once.

Finally let any of $(1)$ through $(5)$ hold. By Theorem~\ref{thm4.11}, $(T,X)$ is thickly transitive, i.e., for all $U,V\in\mathfrak{U}(X)$, $N_T(U,V)$ is a thick subset of $T$. Then no point of $X$ is distal for $(T,X)$ by Corollary~\ref{cor1.11}. Precisely speaking, this follows from a proof similar to that of Corollary~\ref{cor1.11} with Theorem~\ref{thm4.11} in place of Theorem~\ref{thm1.8}.

The proof of Theorem~\ref{thm4.12} is thus completed.
\end{proof}

Note here that comparing with \cite[Theorem~9.13]{Aus} the only new ingredients of Theorem~\ref{thm4.12} are 1): our phase space $X$ is not required to be metrizable; and 2): $T$ is not necessarily a group. The non-metrizable condition of $X$ that is expected by Petersen in \cite{P70} is just the point for Theorem~\ref{thm1.12}.

In particular, let $T=\mathbb{Z}$ and $X$ be a compact metric space, and assume $(T,X)$ has no nontrivial equicontinuous factor; then $(T,X\times X)$ is T.T. (Keynes and Robertson \cite{KR}).

\subsection{Chaos of minimal weak-mixing semiflow}\label{sec4.2}
Let $(T,X)$ be a semiflow on a uniform non-singleton space $(X,\mathscr{U}_X)$, where $\mathscr{U}_X$ is a compatible symmetric uniform structure on $X$. We first introduce the notion of sensitivity.

\begin{defn}[{cf.~\cite{DT}}]
$(T,X)$ is said to be \textit{sensitive to initial conditions} if there exists an $\varepsilon\in\mathscr{U}_X$ such that for all $x\in X$ and $\delta\in\mathscr{U}_X$ there are $y\in\delta[x]$ and $t\in T$ with $(tx,ty)\not\in\varepsilon$.

Given any $\varepsilon\in\mathscr{U}_X$, let
$$\textrm{Equi}_\varepsilon(T,X)=\left\{x\in X\,|\,\exists\,\delta\in\mathscr{U}_X\textit{ s.t. }t(\delta[x])\subseteq\varepsilon[tx]\ \forall t\in T\right\}.$$
Then,
\begin{quote}
$(T,X)$ is not sensitive to initial conditions iff $\textrm{Equi}_\varepsilon(T,X)\not=\emptyset$ for all $\varepsilon\in\mathscr{U}_X$.
\end{quote}
\end{defn}

It is already known that ``weak-mixing is highly non-equicontinuous'' (cf.~Theorem~\ref{thm4.12}). In fact, weak-mixing is even imcompatible with ``$\varepsilon$-equicontinuity'' as follows:

\begin{lem}\label{lem4.14}
If $(T,X)$ is a weak-mixing semiflow on a uniform $T_2$-space $(X,\mathscr{U}_X)$, then $(T,X)$ is sensitive to initial conditions.
\end{lem}

\begin{proof}
By contradiction, suppose $(T,X)$ were not sensitive to initial conditions; then it holds that $\textrm{Equi}_\varepsilon(T,X)\not=\emptyset$ for all $\varepsilon\in\mathscr{U}_X$. By the weak-mixing property, $(T,X\times X)$ is a T.T. semiflow.
Since $X\times X\not=\varDelta_X$, there is an $\varepsilon\in\mathscr{U}_X$ such that $U:=X\times X\setminus\textrm{cls}_{X\times X}\varepsilon\not=\emptyset$. Take an $x_0\in\textrm{Equi}_{\varepsilon/3}(T,X)$ where $\varepsilon/3$ is an entourage in $\mathscr{U}_X$ such that $\varepsilon/3\circ\varepsilon/3\circ\varepsilon/3\subseteq\varepsilon$. Then there is a $\delta\in\mathscr{U}_X$ such that
$t(\delta[x_0]\times\delta[x_0])\subseteq\varepsilon$ for all $t\in T$. Thus $N_T(\delta[x_0]\times\delta[x_0],U)=\emptyset$, which is a contradiction to that $(T,X\times X)$ is a T.T. semiflow.
The proof of Lemma~\ref{lem4.14} is therefore completed.
\end{proof}

Let $\mathscr{K}(T)$ be the collection of non-empty compact subsets of $T$. Given $\varepsilon\in\mathscr{U}_X$ and $x\in X$, we define the $\varepsilon$-stable set of $(T,X)$ at $x$ as follows:
\begin{equation*}
W_\varepsilon^s(T,X;x)=\{y\in X\,|\,\exists\,K\in\mathscr{K}(T)\textit{ s.t. }(tx,ty)\in\varepsilon\ \forall t\in T\setminus K\}.
\end{equation*}

Then there holds the following lemma.

\begin{lem}[{cf.~\cite[Lemma~3.15]{DT}}]\label{lem4.15}
Let $(T,X)$ be a semiflow such that $T$ is $\sigma$-compact. If $(T,X)$ is sensitive to initial conditions, then there is an $\varepsilon\in\mathscr{U}_X$ such that $W_\varepsilon^s(T,X;x)$ is of the first category in $X$ for all $x\in X$.
\end{lem}

The following notion is stronger than the above sensitivity to initial conditions.

\begin{defn}
Let $(T,X)$ be a semiflow and $\varepsilon,\delta\in\mathscr{U}_X$ with $\varepsilon>\delta$. For $x\in X$ define a set
\begin{equation*}
\textrm{LY}_{\varepsilon\textrm{-}\delta}[x]=\left\{y\in X\,|\,y\not\in W_\varepsilon^s(T,X;x)\textit{ and }\exists\,t\in T\textit{ s.t. }(tx,ty)\in\delta\right\}.
\end{equation*}
We say that $(T,X)$ is \textit{$\varepsilon\textrm{-}\delta$ Li-Yorke sensitive} if $\textrm{LY}_{\varepsilon\textrm{-}\delta}[x]$ is of the second category for all $x\in X$.
\end{defn}

\begin{thm}\label{thm4.17}
Let $(T,X)$ be a minimal semiflow with $T$ $\sigma$-compact, which admits of an invariant measure (e.g. $T$ is amenable). If $(T,X)$ is weak-mixing, then there exists an $\varepsilon\in\mathscr{U}_X$ such that for all $\delta\in\mathscr{U}_X$ with $\delta<\varepsilon$, $(T,X)$ is $\varepsilon\textrm{-}\delta$ Li-Yorke sensitive.
\end{thm}

\begin{proof}
By Lemma~\ref{lem4.14}, $(T,X)$ is sensitive to initial conditions. So by Lemma~\ref{lem4.15}, it follows that there exists an $\varepsilon\in\mathscr{U}_X$ such that
$W_\varepsilon^s(T,X;x)$, for all $x\in X$, is of the first category.

Let $\delta\in\mathscr{U}_X$ with $\delta<\varepsilon$ be any given. Given any $x\in X$, let $P[x]$ be the proximal cell at $x$; i.e., $y\in P[x]$ iff $\exists\,z\in X$ and $\{t_n\}\subseteq T$ s.t. $(t_nx,t_ny)\to(z,z)$. Then by a slight modification of the proof of Corollary~\ref{cor1.11} in $\S\ref{sec1.2}$, we can conclude that $P[x]$ is dense in $X$. In addition,
$$
P[x]=\bigcap_{\alpha\in\mathscr{U}_X}\bigcup_{t\in T}t^{-1}(\alpha[tx]).
$$
Then $\bigcup_{t\in T}t^{-1}(\alpha[tx])$ is an open dense subset of $X$ and $\left(\bigcup_{t\in T}t^{-1}(\delta[tx])\right)\setminus W_\varepsilon^s(T,X;x)\subseteq\textrm{LY}_{\varepsilon\textrm{-}\delta}[x]$. Thus, $\textrm{LY}_{\varepsilon\textrm{-}\delta}[x]$ is of the second category.

This thus completes the proof of Theorem~\ref{thm4.17}.
\end{proof}

\begin{defn}
Let $(T,X)$ be a semiflow and $\varepsilon\in\mathscr{U}_X$. For $x\in X$ define a set
\begin{equation*}
\textrm{LY}_{\varepsilon\textrm{-}0}[x]=\left\{y\in X\,|\,y\not\in W_\varepsilon^s(T,X;x)\textit{ and }\exists\,\{t_n\}\in T\textit{ s.t. }{\lim}_n(t_nx,t_ny)\in\varDelta_X\right\}.
\end{equation*}
We say that $(T,X)$ is \textit{$\varepsilon\textrm{-}0$ Li-Yorke sensitive} if $\textrm{LY}_{\varepsilon\textrm{-}0}[x]$ is of the second category for all $x\in X$.
\end{defn}

\begin{thm}[{cf.~\cite{AK} for $T=\mathbb{Z}_+$ and \cite{DT} for $T$ in abelian semigroups}]\label{thm4.19}
Let $(T,X)$ be a minimal semiflow with $T$ $\sigma$-compact and $X$ a compact metric space, which admits of an invariant measure (e.g. $T$ is amenable). If $(T,X)$ is weak-mixing, then there exists an $\varepsilon\in\mathscr{U}_X$ such that $(T,X)$ is $\varepsilon\textrm{-}0$ Li-Yorke sensitive.
\end{thm}

\begin{proof}
This follows from the proof of Theorem~\ref{thm4.17} by noting that $P[x]$ is a dense $G_\delta$-set in $X$ when $X$ is a compact metric space.
\end{proof}

\begin{defn}
Let $T$ be $\sigma$-compact with a sequence of compact subsets $F_1\subset F_2\subset F_3\subset\dotsm$ such that $T=\bigcup F_n$ and $\epsilon>0$. $(T,X)$ with $X$ a metric space is called \textit{densely Li-Yorke $\epsilon$-chaotic} if there is a dense Cantor set $\Theta\subseteq X$  such that for all $x,y\in\Theta$ with $x\not=y$, there are $\{t_n\}$ and $\{s_n\}$ in $T$ with the properties:
\begin{itemize}
\item $\lim_{n\to\infty}d(t_nx,t_ny)=0$;
\item $s_n\not\in F_n$ and $\lim_{n\to\infty}d(s_nx,s_ny)\geq\epsilon$.
\end{itemize}
\end{defn}

Now we can easily obtain a Li-Yorke chaos result from Theorem~\ref{thm4.19}.

\begin{cor}[{cf.~\cite[Proposition~3.18]{DT} for $T$ in abelian semigroups}]\label{cor4.21}
Let $(T,X)$ be a minimal semiflow with $T$ $\sigma$-compact and $X$ a compact metric space, which admits of an invariant measure (e.g. $T$ is amenable). If $(T,X)$ is weak-mixing, then $(T,X)$ is densely Li-Yorke $\epsilon$-chaotic for some $\epsilon>0$.
\end{cor}

\section{Veech's structure theorem for $\mathrm{a.a.}$ semiflows}\label{sec5}
Let $V$ and $D$ be Veech's relations defined as in Definition~\ref{def1.13} in $\S\ref{sec1.3}$.
This section will be devoted to proving (2) and (3) of Theorem~\ref{thm1.15}.

For that, we need to introduce two important lemmas. The first one is borrowed from \cite{Dai}.

\begin{lem}[{cf.~\cite[Theorem~5.11]{Dai}; also \cite[Theorem~1.2]{V68} for $T$ in abelian groups}]\label{lem5.1}
Let $(T,X)$ be an invertible minimal semiflow. Then $V[x]$ is a dense subset of $D[x]$ for each $x\in X$. If $X$ is metrizable, then $V[x]=D[x]$ for all $x\in X$.
\end{lem}

\begin{lem}\label{lem5.2}
Let $(T,X)$ be a bi-continuous semiflow and let $x_0\in P_{\!aa}(T,X)$. Then $V[x_0]$ is a dense subset of $D[x_0]$. If $X$ is metrizable, then $V[x_0]=D[x_0]$.
\end{lem}

\begin{proof}
Let $\Sigma$ be the set of all continuous pseudo-metrics on the compact $\textrm{T}_2$ space $X$. Let $\rho\in\Sigma$ be arbitrarily fixed.
First, for every $\delta>0$ and all $x,y\in X$, we shall say $(x,y)\in\delta$ iff $\rho(x,y)<\delta$. For closed subsets $A,B$ of $X$, we say $\rho_H(A,B)<\delta$ if and only if $\rho(a,b)<\delta$ for all $a\in A$ and $b\in B$. We will say $\rho(A,x)<\delta$ if and only if $\rho(a,x)<\delta$ for some $a\in A$.
Write
\begin{gather*}
\langle T^{-1}\circ T\rangle=\left\{\tau_{i_1}^{-1}\sigma_{i_1}\dotsm\tau_{i_n}^{-1}\sigma_{i_n}\,|\,\tau_{i_k}, \sigma_{i_k}\in T \textit{ for }k=1,\dotsc,n, 1\le n<\infty\right\}.
\end{gather*}
Let $\{\delta_n\}_{n=1}^\infty$ be a sequence of positive numbers such that $\sum_n\delta_n<\infty$.

If $x^\prime\in D[x_0]$, then by definition there exist, for each $n$ and $\epsilon>0$, elements $\sigma,\tau\in N_T(x_0,\delta_n[x_0])$ such that $\rho(\tau^{-1}\sigma x_0,x^\prime)<\epsilon$. Moreover, if $F$ is any finite subset of $\langle T^{-1}\circ T\rangle$ it can also be arranged that $\rho_H(s\gamma x_0,sx_0)<\delta_n\ \forall s\in F$, for $\gamma=\sigma$ and $\tau$.

Since $x_0$ is a distal point so it is an a.p. point, then $N_T(x_0,\delta[x_0])$ is syndetic in $T$ for all $\delta>0$. We can select a sequence $(\sigma_1,\tau_1), (\sigma_2,\tau_2), \dotsc$ inductively as follows. First we can choose $\sigma_1,\tau_1\in N_T(x_0,\delta_1[x_0])$ and $F_0=\{e\}$ with $\rho(\tau_1^{-1}\sigma_1x_0,x^\prime)<\delta_1$. Having chosen $(\sigma_1,\tau_1), \dotsc, (\sigma_n,\tau_n)$ let $F_n$ be the finite set of elements of $\langle T^{-1}\circ T\rangle$ which are representable as
$$
t=\tau_1^{\epsilon_1}\sigma_1^{\epsilon_1^\prime}\dotsm\tau_n^{\epsilon_n}\sigma_n^{\epsilon_n^\prime}\quad \textrm{where }\epsilon_i=0\textrm{ or }-1\textrm{ and } \epsilon_i^\prime=0\textrm{ or }1.
$$
Then choose $\sigma_{n+1},\tau_{n+1}\in N_T(x_0,\delta_{n+1}[x_0])$ in such a way that
\begin{enumerate}
\item[(a)] $\rho_H(s\gamma x_0,sx_0)<\delta_{n+1}\ \forall s\in F_n$, where $\gamma\in\{\sigma_{n+1},\tau_{n+1}\}$; and
\item[(b)] $\rho(\tau_{n+1}^{-1}\sigma_{n+1}x_0,x^\prime)<\delta_{n+1}$.
\end{enumerate}
Based on the sequence $\{(\sigma_n,\tau_n)\}$, we define $\alpha_1=\tau_1, \alpha_2=\sigma_1\tau_2$ in $T$, and in general,
$$
\alpha_n=\sigma_1\dotsm\sigma_{n-1}\tau_n\in T\quad (n=2,3,\dotsc).
$$
If $m<n$ we have from (a) that
\begin{equation*}
\rho(\alpha_mx_0,\alpha_nx_0)\le\sum_{j=0}^{n-m-1}\rho(\alpha_{m+j}x_0,\alpha_{m+j+1}x_0)
\le\sum_{j=0}^{n-m-1}(2\delta_{m+j}+\delta_{m+j+1})
\end{equation*}
tends to $0$ as $m\to\infty$. Thus $\rho\textrm{-}\lim_n\alpha_nx_0$ exists because $\sum_{n=1}^{\infty}\delta_n<\infty$. Letting $y$ be the $\rho$-limit, we now claim $x^\prime\in\rho\textrm{-}\lim_{m\to\infty}\alpha_m^{-1}y$ (that is, $\exists x_m^\prime\in\alpha_m^{-1}y$ s.t. $\rho$-$\lim x_m^\prime=x^\prime$). To see this we note that if $n>m$, then
\begin{equation*}
\alpha_m^{-1}\alpha_nx_0=\tau_m^{-1}\sigma_{m-1}^{-1}\dotsm\sigma_1^{-1}\sigma_1\dotsm\sigma_{n-1}\tau_nx_0=\tau_m^{-1}\sigma_m\dotsm\sigma_{n-1}\tau_nx_0
\end{equation*}
because $t^{-1}tx=x\ \forall t\in T$ for all $x\in P_{\!aa}(T,X)$ and $P_{\!aa}(T,X)$ is invariant; and therefore by argument as above, for $n>m$,
$$
\rho_H\left(\alpha_m^{-1}\alpha_nx_0,\tau_m^{-1}\sigma_mx_0\right)\le\sum_{k=m+1}^n\delta_k\to 0\quad \textrm{as }m\to\infty.
$$
Thus by (b),
\begin{equation*}
\lim_{m\to\infty}\rho\left(\alpha_m^{-1}y,x^\prime\right)=\lim_{m\to\infty}\lim_{n\to\infty}\rho\left(\alpha_m^{-1}\alpha_n x,x^\prime\right)=0
\end{equation*}
and then we can choose $x_m^\prime\in\alpha_m^{-1}y$ such that $x_m^\prime\to x^\prime$. By choosing a subnet $\{\beta_i\}$ from the sequence $\{\alpha_n\}$ in $T$, there are points $z,x^{\prime\prime}\in X$ such that
$\beta_ix_0\to z$ and $x^{\prime\prime}\in V[x_0]\cap\lim_j\lim_i\beta_j^{-1}\beta_ix_0$ in $X$;
and moreover, $\rho(x^{\prime\prime},x^\prime)=0$. Of course, if $\rho$ is just a metric on $X$, then $x^{\prime\prime}=x^\prime$.

This shows that $V[x_0]$ is dense in $D[x_0]$ and $V[x_0]=D[x_0]$ when $X$ is metrizable. The proof of Lemma~\ref{lem5.2} is thus complete.
\end{proof}

The following result is exactly (2) of Theorem~\ref{thm1.15}, which generalizes Veech's \cite[Theorem~1.1]{V68} from abelian groups to invertible semiflows admitting invariant measures by using completely different approaches. In fact, our proof is much more simpler than Veech's one presented in \cite{V68} that depends on harmonic analysis on locally compact abelian groups.

\begin{thm}[{cf.~\cite[(iii) of Theorem~16]{AGN} for $T$ in groups}]\label{5.3}
Let $(T,X)$ be an invertible minimal semiflow admitting an invariant measure. Then $Q[x]=D[x]=\overline{V[x]}$ for all $x\in X$.
\end{thm}

\begin{proof}
Given $x\in X$, $D[x]\subseteq Q[x]$ is obvious. In fact, $V[x]\subseteq Q[x]$ and $Q[x]$ is closed by Definitions~\ref{def1.3} and \ref{def1.13}. So $D[x]=\overline{V[x]}\subseteq Q[x]$ by Lemma~\ref{lem5.1}.
To show $Q[x]\subseteq D[x]$ for all $x\in X$, we fix an $x\in X$ and let $y\in Q[x]$. Then by (3) of Lemma~\ref{lem3.7}, there are nets $\{y_n\}, \{z_n\}$ in $X$ and $\{t_n\}$ in $T$ such that
\begin{gather*}
y_n\to y,\quad z_n\to x,\quad z_n=t_ny_n,\quad t_nx\to x.
\end{gather*}
Then $Q[x]\subseteq D[x]$ by Lemma~\ref{lem3.12}.
Therefore, we have concluded that $D[x]=Q[x]$ for all $x\in X$. This proves Theorem~\ref{5.3}.
\end{proof}

Using Lemma~\ref{lem5.2} in place of Lemma~\ref{lem5.1}, we can obtain the following, which is important for (3) of Theorem~\ref{thm1.15} stated in $\S\ref{sec1.3}$.

\begin{lem}\label{5.4}
Let $(T,X)$ be a minimal bi-continuous semiflow admitting an invariant measure. If $x_0\in P_{\!aa}(T,X)$, then $\{x_0\}=V[x_0]=D[x_0]=Q[x_0]$.
\end{lem}

\begin{proof}
At first, by \cite[Lemma~3.5]{AD} we see $(T,X)$ is surjective. By Lemma~\ref{lem5.2}, it follows that $\{x_0\}=V[x_0]=D[x_0]\subseteq Q[x_0]$. Moreover, $Q[x_0]\subseteq D[x_0]$ by Lemma~\ref{lem3.12}. Thus we can conclude that $\{x_0\}=V[x_0]=D[x_0]=Q[x_0]$ and the proof is complete.
\end{proof}

Based on Theorem~\ref{thm1.4} and Lemma~\ref{5.4} we can obtain the next result which is just the third part of Theorem~\ref{thm1.15} stated in $\S\ref{sec1.3}$.

\begin{thm}\label{thm5.5}
Let $(T,X)$ be a minimal bi-continuous semiflow admitting an invariant measure. Then
$(T,X)$ is $\mathrm{a.a.}$ if and only if $\pi\colon(T,X)\rightarrow(T,X_\textit{eq})$ is of almost 1-1 type.
\end{thm}

\begin{proof}
Since $(T,X)$ is a minimal semiflow admitting an invariant measure, by Theorem~\ref{thm1.4} we have $S_\textit{eq}(X)=Q(X)$.

(1). Let $(T,X)$ be a.a.; then there exists an $x_0\in X$ such that $V[x_0]=\{x_0\}$. Then $Q[x]=\{x_0\}$ by Lemma~\ref{5.4}. Thus, $\pi\colon(T,X)\rightarrow(T,X_\textit{eq})$ is of almost 1-1 type at $x_0$. This has concluded the ``only if'' part.

(2). Assume $\pi\colon (T,X)\rightarrow(T,X_\textit{eq})$ is of almost 1-1 type. Then we can take an $x\in X$ such that $\pi^{-1}\pi(x)=\{x\}$. Set $y=\pi(x)$. Since $(T,X_\textit{eq})$ is minimal equicontinuous invertible by Lemmas~\ref{lem2.4A} and \ref{lem2.2}, hence $V[y]=\{y\}$. This implies that $V[x]=\{x\}$ and so $(T,X)$ is an a.a. semiflow. This proves the ``if'' part. The proof is complete.
\end{proof}

\begin{cor1.17}
Let $(T,X)$ be a minimal bi-continuous semiflow, which admits an invariant measure. Then $(T,X)$ is equicontinuous iff all points are a.a. points.
\end{cor1.17}

\begin{proof}
(1) The ``only if'' part. By Lemma~\ref{lem2.3}, $P(X)=Q(X)=\varDelta_X$. Thus $(T,X)$ is pointwise a.a. by Theorem~\ref{thm1.15}.

(2) The ``if'' part. By Theorem~\ref{thm1.15}, $Q(X)=\varDelta_X$. Then $(T,X)$ is equicontinuous by Lemma~\ref{lem2.3}. This proves Corollary~\ref{cor1.17}.
\end{proof}
\section*{\textbf{Acknowledgments}}%
The author would like to thank Professor Joe~Auslander and Professor Eli Glasner for their many helpful comments on the first version of this paper.

This work was partly supported by National Natural Science Foundation of China (Grant Nos. 11431012, 11271183) and PAPD of Jiangsu Higher Education Institutions.

\end{document}